\documentclass[11pt, letterpaper, oneside,reqno]{amsart}

\usepackage{latexsym, amsmath, amssymb, amsfonts, amscd,bm}
\usepackage{mathtools}
\usepackage{amsthm}
\usepackage{t1enc}
\usepackage[mathscr]{eucal}
\usepackage{indentfirst}
\usepackage{graphicx, pb-diagram}
\usepackage{fancyhdr}
\usepackage{fancybox}
\usepackage{enumerate}
\usepackage{xcolor}
\usepackage[all]{xy}
\usepackage{hyperref}
\usepackage{tikz-cd}
\usepackage{blkarray}
\usepackage{url}
\usepackage{float}
\usepackage{bm}
\usepackage[makeroom]{cancel}
\usepackage{accents}

\theoremstyle{plain}
\newtheorem{thm}{Theorem}[section]

\newtheorem*{pbl}{Problem}
\newtheorem{prop}[thm]{Proposition}

\newtheorem{lemma}[thm]{Lemma}
\newtheorem{cor}[thm]{Corollary}

\newtheoremstyle{underline}% name
{}        % Space above, empty = `usual value'
{}              % Space below
{}              % Body font
{}    % Indent amount (empty = no indent, \parindent = para indent)
{\large}              % Thm head font
{:}             % Punctuation after thm head
{1mm}         % Space after thm head: \newline = linebreak
{{\underline{\thmname{#1}\thmnumber{ #2}}}}  % Thm head spec

\theoremstyle{underline}
\newtheorem*{claim*}{Claim}

\theoremstyle{definition}
\newtheorem{defi}[thm]{Definition}

\theoremstyle{remark}
\newtheorem{remark}[thm]{Remark}
\newtheorem{ex}[thm]{Example}

\newtheorem*{ack}{Acknowledgements}

\newcommand*{\dt}[1]{\accentset{\mbox{\large\bfseries .}}{#1}}

\definecolor{forest}{rgb}{0,0.5,0}

\begin{document}
	
	\title{Morse-Novikov cohomology and rigidity of Lie affine foliations}
	\author{Stephane Geudens}
	\address{{\scriptsize Institute of Mathematics, Polish Academy of Sciences,  ul. Sniadeckich 8, 00-656 Warsaw, Poland}}
	\email{stephane\_geudens@hotmail.com}
	
	\begin{abstract}  
		In \cite{Lie}, a cohomological criterion is given for a Lie $\mathfrak{g}$-foliation on a compact manifold to be rigid among nearby Lie foliations. Our aim is to look for examples of this rigidity statement in case the Lie foliation is modeled on the two-dimensional non-abelian Lie algebra $\mathfrak{g}=\mathfrak{aff}(1)$. We study the relevant cohomology group in detail, showing that it can be expressed in terms of Morse-Novikov cohomology. We find the precise conditions under which it vanishes, which yields many examples of rigid Lie affine foliations. In particular, we show that any Lie affine foliation on a compact, connected, orientable manifold of dimension $3$ or $4$ is rigid when deformed as a Lie foliation. Our results rely on a computation of the Morse-Novikov cohomology groups associated with a nowhere-vanishing closed one-form with discrete period group, which may be of independent interest.
	\end{abstract}
	
	\maketitle
	
	\setcounter{tocdepth}{1} %doesn't display subsections in TOC 
	\tableofcontents
	
	\section*{Introduction}
	
	Given a manifold $M$, a Lie foliation modeled on a Lie algebra $\mathfrak{g}$ is given by the kernel of a non-singular Maurer-Cartan form $\omega\in\Omega^{1}(M,\mathfrak{g})$. Lie foliations are a particularly important class of Riemannian foliations, since Molino's structure theory \cite{Molino} reduces the study of Riemannian foliations on compact manifolds to that of Lie foliations.	This paper is motivated by the work \cite{Lie}, which studies deformations of a Lie foliation $(\mathcal{F},\omega)$ within the class of Lie foliations on $M$. The precise setup of this deformation problem will be recalled in Def.~\ref{def:Lie-fol}. We would like to stress that one deforms the defining Maurer-Cartan form $\omega$, rather than just the underlying foliation $\mathcal{F}$.
	One of the main results in \cite{Lie} is a rigidity theorem for Lie foliations on compact manifolds, which states that $(\mathcal{F},\omega)$ is rigid among nearby Lie foliations provided that a suitable cohomology group $H^{1}(\mathcal{A})$ vanishes. Our aim is to look for situations in which this result applies. Since the authors of \cite{Lie} study the cohomology $H^1(\mathcal{A})$ in detail for Lie foliations modeled on abelian Lie algebras $\mathfrak{g}$, we will focus on the case in which $\mathfrak{g}=\mathfrak{aff}(1)$ is the simplest non-abelian Lie algebra. We refer to such Lie foliations as \emph{Lie affine foliations}. They are defined by a pair of one-forms $\omega_1,\omega_2\in\Omega^{1}(M)$ that are linearly independent everywhere and satisfy
	\begin{equation}\label{eq:defining-equations}
		\begin{cases}
			d\omega_1-\omega_2\wedge\omega_1=0\\
			d\omega_2=0
		\end{cases}.
	\end{equation}	
	From these equations, it is apparent that the study of Lie affine foliations involves a classical cohomology theory called \emph{Morse-Novikov cohomology}, which is defined as follows. A choice of closed one-form $\theta\in\Omega^{1}(M)$ gives a differential
	\[
	d_{\theta}:\Omega^{\bullet}(M)\rightarrow\Omega^{\bullet+1}(M):\alpha\mapsto d\alpha-\theta\wedge\alpha,
	\]
	which is defined by twisting the de Rham differential with $\theta$. The resulting cohomology groups $H^{\bullet}_{\theta}(M)$ are called the Morse-Novikov cohomology groups associated with $\theta$. This cohomology theory goes back to the work of Lichnerowicz \cite{Lichnerowicz} and Novikov \cite{Novikov} and has been studied intensively because of its relevance to locally conformal symplectic geometry \cite{Haller}, \cite{Ornea} and Morse theory for closed one-forms \cite{Pajitnov}. Notable results that have been obtained to date include vanishing theorems for $H^{\bullet}_{\theta}(M)$ under conditions of Riemannian nature \cite{Leon}, \cite{chen} and computations of $H^{\bullet}_{\theta}(M)$ in particular classes of examples \cite{Otiman}, \cite{Istrati}, \cite{Banyaga}.
	
	\bigskip
	
	Section~\ref{sec:one} is devoted to a result concerning Morse-Novikov cohomology that may be of independent interest. Towards studying rigidity of Lie affine foliations, we will compute the Morse-Novikov cohomology groups $H^{\bullet}_{\theta}(M)$ when the manifold $M$ is compact and connected and the closed one-form $\theta$ is nowhere-vanishing with discrete period group. Our approach relies on the fact that there is a diffeomorphism between $M$ and a suitable mapping torus $\mathbb{T}_{\varphi}$, which takes the foliation $\mathcal{F}$ on $M$ defined by the kernel of $\theta$ to the foliation on $\mathbb{T}_{\varphi}$ by fibers of $\mathbb{T}_{\varphi}\rightarrow S^{1}$. As the Morse-Novikov cohomology groups $H^{\bullet}_{\theta}(M)$ only depend on the de Rham cohomology class of $\theta$, this reduces the computation of $H^{\bullet}_{\theta}(M)$ to that of $H^{\bullet}_{cdt}(\mathbb{T}_{\varphi})$. Here $c\in\mathbb{R}_{0}$ is a non-zero constant and $t$ is the angle coordinate on $S^{1}$. We then prove the following result (see Cor.~\ref{cor:computed}). Below, we use the notation
	\[
	K^{\bullet}_c:=\ker\left([\varphi^{*}-e^{-c}\mathrm{Id}]:H^{\bullet}(L)\rightarrow H^{\bullet}(L)\right)
	\]
	and
	\[
	C^{\bullet}_c:=\mathrm{coker}\left([\varphi^{*}-e^{-c}\mathrm{Id}]:H^{\bullet}(L)\rightarrow H^{\bullet}(L)\right).
	\]
	
	\vspace{0.2cm}
	\noindent
	\textbf{Proposition A.}	
	\emph{
		The Morse-Novikov cohomology $H^{\bullet}_{cdt}(\mathbb{T}_{\varphi})$ for $c\neq 0$ is given as follows.
		\begin{enumerate}[i)]
			\item $H^{0}_{cdt}(\mathbb{T}_{\varphi})=0$,
			\item For $k\geq 1$, we have a short exact sequence
			\begin{align*}
				0\longrightarrow C_{c}^{k-1}\longrightarrow H^{k}_{cdt}(\mathbb{T}_{\varphi})\longrightarrow K_{c}^{k}\longrightarrow 0.
			\end{align*}
		\end{enumerate}
	}
	
	\vspace{0.2cm}
	\noindent
	This result unifies some computations for particular examples that have appeared in the literature \cite{Otiman}, \cite{Moroianu}. The latter rely on commonly used tools like the Mayer-Vietoris sequence. The novelty of our result lies in the fact that we use a geometric approach based on foliation theory. Indeed, in Cor.~\ref{cor:sec} we relate the Morse-Novikov cohomology $H^{\bullet}_{\theta}(M)$ for a nowhere-vanishing closed one-form $\theta\in\Omega^{1}(M)$ with the foliated cohomology $H^{\bullet}(\mathcal{F})$ of the foliation $\mathcal{F}$ defined by $\theta$. This result can be used to compute the Morse-Novikov cohomology $H^{\bullet}_{\theta}(M)$ whenever the foliated cohomology $H^{\bullet}(\mathcal{F})$ can be calculated. It applies in particular when the foliation $\mathcal{F}$ is given by the fibers of the mapping torus $\mathbb{T}_{\varphi}\rightarrow S^{1}$, which yields Proposition~A.
	
	\bigskip
	Section~\ref{sec:deformations} is devoted to rigidity of Lie affine foliations. We start by expressing the cohomology group $H^{1}(\mathcal{A})$ in more classical terms. The equations \eqref{eq:defining-equations} say that the one-form $\omega_2\in\Omega^{1}(M)$ is closed and that $\omega_1\in\Omega^{1}(M)$ is closed for the Morse-Novikov differential $d_{\omega_2}$. It follows that we get a cochain map 
	\[
	\Phi:\big(\Omega^{\bullet}(M),d\big)\rightarrow\big(\Omega^{\bullet+1}(M),d_{\omega_2}\big):\alpha\mapsto\alpha\wedge\omega_1.
	\]
	We then show that the cohomology group $H^{1}(\mathcal{A})$ is isomorphic to the cohomology $H^{1}(\mathcal{C}(\Phi))$ of the mapping cone $\mathcal{C}(\Phi)$. This can be used to deduce precise criteria for the vanishing of $H^{1}(\mathcal{A})$, which yields the following rigidity result (see Cor.~\ref{cor:req-rigidity}).

	\vspace{0.3cm}
	\noindent
	\textbf{Proposition B.}
	\emph{
		Let $M$ be compact and connected, with a Lie $\mathfrak{aff}(1)$-foliation $(\mathcal{F},\omega)$ defined by one-forms $\omega_1,\omega_2\in\Omega^{1}(M)$ satisfying
		\begin{equation*}
			\begin{cases}
				d\omega_1-\omega_2\wedge\omega_1=0\\
				d\omega_2=0
			\end{cases}.
		\end{equation*}
		If $H^{1}_{\omega_2}(M)=\mathbb{R}[\omega_1]$ and $\bullet\wedge[\omega_1]:H^{1}(M)\rightarrow H^{2}_{\omega_2}(M)$ is injective, then $(\mathcal{F},\omega)$ is rigid.
	}
	
	\vspace{0.3cm}
	\noindent
	To obtain examples of this result, we restrict to a particular class of Lie affine foliations on mapping tori which we call \emph{model Lie affine foliations}. They are constructed as follows. Let $L$ be a compact, connected manifold and $\varphi\in\text{Diff}(L)$ a diffeomorphism. Pick a closed nowhere-vanishing one-form $\alpha\in\Omega^{1}(L)$ satisfying $\varphi^{*}\alpha=e^{\lambda}\alpha$ for some $\lambda\neq 0$. We get a Lie affine foliation $(\mathcal{G}_{\alpha,\lambda},\omega)$ on the mapping torus $\mathbb{T}_{\varphi}$, defined by the one-forms
	\begin{equation}\label{eq:defining-forms}
		\begin{cases}
			\omega_1:=e^{-\lambda t}\alpha\\
			\omega_2=-\lambda dt
		\end{cases}.
	\end{equation}
	We use Proposition A to analyze the cohomological requirements in Proposition B for model Lie affine foliations. We obtain sufficient conditions under which these requirements are met. This gives a concrete way to construct examples of rigid Lie affine foliations (see Cor.~\ref{cor:rigidity-model}).
	
	\vspace{0.3cm}
	\noindent
	\textbf{Corollary C.}
	\emph{ A model Lie affine foliation $(\mathcal{G}_{\alpha,\lambda},\omega)$ is rigid under the following conditions:
		\begin{enumerate}
			\item The eigenvalue $e^{\lambda}$ of $[\varphi^{*}]:H^{1}(L)\rightarrow H^{1}(L)$ has algebraic multiplicity equal to $1$,
			\item The map $\bullet\wedge[\alpha]:\ker\left([\varphi^{*}-\mathrm{Id}]:H^{1}(L)\rightarrow H^{1}(L)\right)\rightarrow H^{2}(L)$ is injective.
		\end{enumerate}
	}
	
	\vspace{0.1cm}
	\noindent
	Model Lie affine foliations are of particular importance because any Lie affine foliation on a compact, connected, orientable manifold of dimension $3$ or $4$ is foliated diffeomorphic to a foliated manifold $(\mathbb{T}_{\varphi},\mathcal{G}_{\alpha,\lambda})$ of this type \cite{caron},\cite{Matsumoto}. Making use of these normal forms, Corollary C allows us to prove the following (see Cor.~\ref{cor:rigid3} and Cor.~\ref{cor:rigid4}).
	
	\vspace{0.3cm}
	\noindent
	\textbf{Corollary D.}
	\emph{
		Let $M$ be a compact, connected and orientable $3$- or $4$-manifold.  Any Lie affine foliation $(\mathcal{F},\eta)$ on $M$ is rigid, when deformed as a Lie foliation.
	}
	
	\vspace{0.3cm}
	\noindent
	To prove Corollary D, a certain technical difficulty needs to be overcome. By assumption, the Lie affine foliation $\mathcal{F}$ on $M$ is defined by one-forms $\eta_1,\eta_2\in\Omega^{1}(M)$. The existence of a foliated diffeomorphism $\phi:(\mathbb{T}_{\varphi},\mathcal{G}_{\alpha,\lambda})\rightarrow (M,\mathcal{F})$ does not imply that $\phi^{*}\eta_1$ and $\phi^{*}\eta_2$ agree with $\omega_1,\omega_2$ defined in \eqref{eq:defining-forms} -- these pairs of one-forms merely define the same foliation $\mathcal{G}_{\alpha,\lambda}$ on $\mathbb{T}_{\varphi}$. However, we prove that for model Lie affine foliations $(\mathcal{G}_{\alpha,\lambda},\omega)$, the cohomological requirements of Proposition B are in fact independent of the defining Maurer-Cartan form $\omega$ (see Prop.~\ref{prop:other-forms} and Cor.~\ref{cor:equiv-rigidity}). That is to say, they only depend on the underlying foliation $\mathcal{G}_{\alpha,\lambda}$. This is why the mere existence of a foliated diffeomorphism $\phi:(\mathbb{T}_{\varphi},\mathcal{G}_{\alpha,\lambda})\rightarrow (M,\mathcal{F})$ allows us to conclude that the Maurer-Cartan form $\eta$ defining $\mathcal{F}$ is rigid.

	\begin{ack}	
		During the preparation of this article, the author received support from the UCL Institute for Mathematical and Statistical Sciences (IMSS) and from the Institute of Mathematics of the Polish Academy of Sciences (IMPAN). He would like to thank the referee for useful comments and suggestions which improved the quality of this article.
	\end{ack}

	\section{Results on Morse-Novikov cohomology}\label{sec:one}
	
	In this section, we compute the Morse-Novikov cohomology groups of a compact, connected manifold associated with a nowhere-vanishing closed one-form whose group of periods is discrete. 
	Our motivation is to study the deformation theory of Lie affine foliations, which relies heavily on these cohomology groups as we will see in \S\ref{sec:deformations}.

	\subsection{Preliminaries}
	We recall the definition of the Morse-Novikov cohomology groups and some basic properties that will be useful in the sequel.
	
	\begin{defi}
		Let $M$ be a manifold and $\theta\in\Omega^{1}(M)$ a closed one-form. Twisting the de Rham differential $d$ by $\theta$ gives a new differential $d_{\theta}$, defined by
		\[
		d_{\theta}:\Omega^{k}(M)\rightarrow\Omega^{k+1}(M):\alpha\mapsto d\alpha-\theta\wedge\alpha.
		\]
		The associated cohomology groups $H^{\bullet}_{\theta}(M)$ are the Morse-Novikov cohomology groups.
	\end{defi}
	
	The following properties are well-known and can be found for instance in \cite{Haller}.
	
	\begin{prop}\label{prop:prelim}
		$i)$ If $\theta'=\theta+df$ for some $f\in C^{\infty}(M)$, then we get a cochain isomorphism
		\[
		\big(\Omega^{\bullet}(M),d_{\theta'}\big)\rightarrow\big(\Omega^{\bullet}(M),d_{\theta}\big):\alpha\mapsto e^{-f}\alpha.
		\]
		Hence, the cohomology $H^{\bullet}_{\theta}(M)$ only depends on the de Rham cohomology class $[\theta]\in H^{1}(M)$.
		
		\vspace{0.1cm}
		$ii)$ If $M$ is connected and $\theta$ is not exact, then $H^{0}_{\theta}(M)$ vanishes.
	\end{prop}
	
	The following subsections \S\ref{subsec:exact-seq} and \S\ref{subsec:non-singular} are dedicated to computing the Morse-Novikov cohomology $H^{\bullet}_{\theta}(M)$ under the assumption that $M$ is compact and connected and that the closed one-form $\theta\in\Omega^{1}(M)$ is nowhere-vanishing with discrete group of periods.

	\subsection{An exact sequence}\label{subsec:exact-seq}
	Throughout this subsection, let $M$ be any manifold endowed with a nowhere-vanishing closed one-form $\theta\in\Omega^{1}(M)$. 
	Let $\mathcal{F}$ be the foliation whose tangent distribution is the kernel of $\theta$. In the following, we will describe the Morse-Novikov cohomology $H^{\bullet}_{\theta}(M)$ in terms of data attached to the foliation $\mathcal{F}$.
	We will denote by $\big(\Omega^{\bullet}(\mathcal{F}),d_{\mathcal{F}}\big)$ the foliated de Rham complex, consisting of leafwise differential forms $\Omega^{\bullet}(\mathcal{F}):=\Gamma(\wedge^{\bullet}T^{*}\mathcal{F})$ endowed with the leafwise de Rham differential $d_{\mathcal{F}}$. The latter is defined by the usual Koszul formula, i.e. for $\alpha\in\Omega^{k}(\mathcal{F})$ and $V_0,\ldots,V_k\in\Gamma(T\mathcal{F})$ we have
	\begin{align*}
		d_{\mathcal{F}}\alpha(V_0,\ldots,V_k)&=\sum_{i=0}^{k}(-1)^{i}V_i\big(\alpha(V_0,\ldots,V_{i-1},\widehat{V_i},V_{i+1},\ldots,V_k)\big)\\
		&\hspace{0.5cm}+\sum_{i<j}(-1)^{i+j}\alpha\big([V_i,V_j],V_0,\ldots,\widehat{V_i},\ldots,\widehat{V_j},\ldots,V_k\big).
	\end{align*}
	The associated cohomology $H^{\bullet}(\mathcal{F})$ is called the foliated cohomology of $\mathcal{F}$. The next lemma relates the Morse-Novikov cohomology $H^{\bullet}_{\theta}(M)$ with the foliated cohomology $H^{\bullet}(\mathcal{F})$.
	
	\begin{lemma}
		The Morse-Novikov complex fits in a canonical short exact sequence 
		\begin{equation}\label{eq:short-exact}
		0\longrightarrow\big(\Omega^{\bullet-1}(\mathcal{F}),d_{\mathcal{F}}\big)\overset{j}{\longrightarrow}\big(\Omega^{\bullet}(M),d_{\theta}\big)\overset{r}{\longrightarrow}\big(\Omega^{\bullet}(\mathcal{F}),d_{\mathcal{F}}\big)\longrightarrow 0.
		\end{equation}
		The maps appearing above are defined as follows:
		\begin{enumerate}
			\item The map $r:\Omega^{\bullet}(M)\rightarrow \Omega^{\bullet}(\mathcal{F})$ is the restriction to the leaves of $\mathcal{F}$.
			\item The map $j:\Omega^{\bullet-1}(\mathcal{F})\rightarrow\Omega^{\bullet}(M)$ is defined by setting
			\[
			j(\alpha)=\widetilde{\alpha}\wedge\theta,
			\]
			where $\widetilde{\alpha}\in\Omega^{\bullet-1}(M)$ is any extension of $\alpha\in\Omega^{\bullet-1}(\mathcal{F})$.
		\end{enumerate}
	\end{lemma}
	\begin{proof}
		First, it is clear that the restriction $r:\big(\Omega^{\bullet}(M),d\big)\rightarrow\big(\Omega^{\bullet}(\mathcal{F}),d_{\mathcal{F}}\big)$ is a surjective cochain map. The same then holds when replacing $\big(\Omega^{\bullet}(M),d\big)$ by $\big(\Omega^{\bullet}(M),d_{\theta}\big)$, as $r(\theta)=0$.
		
		Next, note that the map $j$ is well-defined. Indeed, if $\widetilde{\alpha},\widetilde{\alpha}'\in\Omega^{k-1}(M)$ are two extensions of $\alpha\in\Omega^{k-1}(\mathcal{F})$ then the difference $\widetilde{\alpha}'-\widetilde{\alpha}$ vanishes along the leaves of $\mathcal{F}$, hence it factors through $\theta$. Consequently, the wedge product $(\widetilde{\alpha}'-\widetilde{\alpha})\wedge\theta$ vanishes. Moreover, the map $j$ is injective since vanishing of $\widetilde{\alpha}\wedge\theta$ implies that the restriction of $\widetilde{\alpha}$ to the leaves of $\mathcal{F}$ vanishes. To see that $j$ is a cochain map, we compute
		\begin{align*}
			d_{\theta}(j(\alpha))&=d(\widetilde{\alpha}\wedge\theta)-\theta\wedge\widetilde{\alpha}\wedge\theta\\
			&=d\widetilde{\alpha}\wedge\theta\\
			&=\widetilde{d_{\mathcal{F}}\alpha}\wedge\theta\\
			&=j(d_{\mathcal{F}}\alpha).
		\end{align*}
		In the third equality, we used that $d\widetilde{\alpha}$ and $\widetilde{d_{\mathcal{F}}\alpha}$ have the same restriction to the leaves of $\mathcal{F}$, namely $d_{\mathcal{F}}\alpha$. This confirms that $j$ is an injective cochain map. 
		
		At last, the image of $j$ consists of the differential forms that factor through $\theta$. These are exactly the forms that restrict to zero along the leaves of $\mathcal{F}$, i.e. the kernel of $r$.
	\end{proof}
	
	Consequently, we get a long exact sequence in cohomology
	\begin{equation}\label{eq:long-ex}
		\cdots\rightarrow H^{k-1}(\mathcal{F})\overset{[j]}{\rightarrow} H^{k}_{\theta}(M)\overset{[r]}{\rightarrow} H^{k}(\mathcal{F})\overset{\partial}{\rightarrow} H^{k}(\mathcal{F})\overset{[j]}{\rightarrow} H^{k+1}_{\theta}(M)\overset{[r]}{\rightarrow} H^{k+1}(\mathcal{F})\rightarrow\cdots
	\end{equation}

	To describe the connecting homomorphism $\partial$ explicitly, we will fix a suitable vector field $V\in\mathfrak{X}(M)$ that is $\mathcal{F}$-projectable. Let us first recall the definition of this notion.
	
	\begin{defi}\label{def:projectable}
	A vector field $V\in\mathfrak{X}(M)$ is called $\mathcal{F}$-projectable if $[V,\Gamma(T\mathcal{F})]\subset\Gamma(T\mathcal{F})$.
	\end{defi}
	
	These are exactly the vector fields $V\in\mathfrak{X}(M)$ which locally project to the leaf space $M/\mathcal{F}$. Equivalently, the flow of $V$ takes leaves to leaves. Consequently, there is a well-defined Lie derivative $\pounds_V:\Omega^{\bullet}(\mathcal{F})\rightarrow \Omega^{\bullet}(\mathcal{F})$ if $V$ is $\mathcal{F}$-projectable. We spell this out in the next remark.
	
	\begin{remark}\label{rem:Lie-der}
	Let $V\in\mathfrak{X}(M)$ be an $\mathcal{F}$-projectable vector field. For $\alpha\in\Omega^{k}(\mathcal{F})$, we can define the Lie derivative $\pounds_{V}\alpha$ by picking any extension $\widetilde{\alpha}\in\Omega^{k}(M)$ of $\alpha$ and setting
	\begin{equation}\label{eq:Lie-derivative}
	\pounds_{V}\alpha:=r\big(\pounds_{V}\widetilde{\alpha}\big).
	\end{equation}
	To show that this does not depend on the choice of extension, we check that if $\beta\in\Omega^{k}(M)$ vanishes when restricted to $\mathcal{F}$, then the same holds for $\pounds_{V}\beta$. This is indeed the case, since for $X_1,\ldots,X_{k}\in\Gamma(T\mathcal{F})$ we have that
	\[
	\big(\pounds_{V}\beta\big)(X_1,\ldots,X_k)=V\big(\beta(X_1,\ldots,X_k)\big)-\sum_{i=1}^{k}\beta\big(X_1,\ldots,X_{i-1},[V,X_i],X_{i+1},\ldots,X_k\big).
	\] 
	Here we make crucial use of the fact that $[V,X_i]\in\Gamma(T\mathcal{F})$, because $V$ is $\mathcal{F}$-projectable.
	\end{remark}
	
	Returning to the long exact sequence \eqref{eq:long-ex}, fix a vector field $V\in\mathfrak{X}(M)$ satisfying $\theta(V)=1$. Note that $V$ is an $\mathcal{F}$-projectable vector field, since for $X\in\Gamma(T\mathcal{F})$ we have that
	\begin{align*}
	0=(d\theta)(V,X)
	=V\big(\theta(X)\big)-X\big(\theta(V)\big)-\theta\big([V,X]\big)
	=-\theta\big([V,X]\big).
	\end{align*}
	Hence, we have a well-defined Lie derivative $\pounds_V:\Omega^{\bullet}(\mathcal{F})\rightarrow \Omega^{\bullet}(\mathcal{F})$ defined by equation \eqref{eq:Lie-derivative}. It shows up in the following result, which describes the connecting homomorphism $\partial$ in \eqref{eq:long-ex}.

	\begin{lemma}
		The connecting homomorphism $\partial$ in the long exact sequence \eqref{eq:long-ex} reads
		\[
		\partial:H^{k}(\mathcal{F})\rightarrow H^{k}(\mathcal{F}):[\alpha]\mapsto (-1)^{k}\left[\pounds_{V}\alpha-\alpha\right].
		\]
	\end{lemma}
	\begin{proof}
	Let us first recall the abstract definition of the connecting homomorphism $\partial$. Given a class  $[\alpha]\in H^{k}(\mathcal{F})$, pick a representative $\alpha\in\Omega^{k}(\mathcal{F})$. Take any preimage of $\alpha$ under the surjection in the short exact sequence \eqref{eq:short-exact}, i.e. any extension $\widetilde{\alpha}\in\Omega^{k}(M)$ of $\alpha$. Then $d_{\theta}\widetilde{\alpha}$ lies in the kernel of the surjection in \eqref{eq:short-exact}, which is the image of the inclusion in \eqref{eq:short-exact}. Hence, 
		\begin{equation}\label{eq:help}
			d\widetilde{\alpha}-\theta\wedge\widetilde{\alpha}=\beta\wedge\theta
		\end{equation}
		for some $\beta\in\Omega^{k}(M)$. By definition, the connecting homomorphism is given by $\partial[\alpha]=[r(\beta)]$.
		
		To make this explicit, we use the vector field $V\in\mathfrak{X}(M)$ we fixed before with the property that $\theta(V)=1$. Contracting the equality \eqref{eq:help} with $V$ and using that $r(\theta)$ vanishes, we obtain
		\begin{align*}
			[r(\beta)]&=(-1)^{k}\left[r\left(\iota_V(\beta\wedge\theta)\right)\right]\\
			&=(-1)^{k}\left[r\left(\iota_V(d\widetilde{\alpha}-\theta\wedge\widetilde{\alpha})\right)\right]\\
			&=(-1)^{k}\left[r\left(\iota_{V}d\widetilde{\alpha}\right)-\alpha\right]\\
			&=(-1)^{k}\left[r\left(\pounds_{V}\widetilde{\alpha}-d\iota_{V}\widetilde{\alpha}\right)-\alpha\right]\\
			&=(-1)^{k}\left[r\left(\pounds_{V}\widetilde{\alpha}\right)-d_{\mathcal{F}}\big(r(\iota_{V}\widetilde{\alpha})\big)-\alpha\right]\\
			&=(-1)^{k}\left[\pounds_{V}\alpha-\alpha\right].
		\end{align*}
		The last equality uses the definition \eqref{eq:Lie-derivative} of the Lie derivative. This proves the statement.
	\end{proof}

	The long exact sequence \eqref{eq:long-ex} can be split into short exact sequences. Below, we will denote
	\begin{align*}
	&K_{\partial}^{\bullet}:=\ker\left(\partial:H^{\bullet}(\mathcal{F})\rightarrow H^{\bullet}(\mathcal{F})\right),\\
	&C_{\partial}^{\bullet}:=\text{coker}\left(\partial:H^{\bullet}(\mathcal{F})\rightarrow H^{\bullet}(\mathcal{F})\right).
	\end{align*}
	
	\begin{cor}\label{cor:sec}
		Let $\theta\in\Omega^{1}(M)$ be a nowhere-vanishing closed one-form and $\mathcal{F}$ the foliation defined by $\theta$. Then the Morse-Novikov cohomology fits in a short exact sequence
		\begin{equation}\label{eq:sh-ex}
			0\longrightarrow C_{\partial}^{k-1}\longrightarrow H^{k}_{\theta}(M)\overset{[r]}{\longrightarrow} K_{\partial}^{k}\longrightarrow 0.
		\end{equation}
	\end{cor}
	
	Cor.~\ref{cor:sec} can be used to compute the Morse-Novikov cohomology $H^{\bullet}_{\theta}(M)$ in situations where one can calculate the foliated cohomology $H^{\bullet}(\mathcal{F})$. In the next subsection, we consider such a situation, namely when the underlying manifold $M$ is compact, connected and the period group of the one-form $\theta$ is discrete. The computation of Morse-Novikov cohomology in that setup will be used in \S \ref{sec:deformations} to study the deformation theory of Lie affine foliations.

	\subsection{Non-singular one-forms with discrete period group}\label{subsec:non-singular}
	From now on, assume that the manifold $M$ is compact and connected. As before, let $\theta\in\Omega^{1}(M)$ be a closed nowhere-vanishing one-form. We have in particular that the de Rham cohomology class $[\theta]\in H^{1}(M)$ is non-zero. Therefore, the Morse-Novikov cohomology groups $H^{\bullet}_{\theta}(M)$ are in general different from the de Rham cohomology groups $H^{\bullet}(M)$. 
	As before, let $\mathcal{F}$ denote the foliation integrating $\ker\theta$. It is known that the properties of $\mathcal{F}$ depend heavily on the period group
	\[
	\text{Per}(\theta):=\left\{\int_{\gamma}\theta: [\gamma]\in\pi_1(M)\right\}.
	\]
	Since $\text{Per}(\theta)$ is a subgroup of $(\mathbb{R},+)$, it is either discrete or dense. In the former case, the foliation $\mathcal{F}$ is given by the fibers of a suitable fibration $M\rightarrow S^{1}$. In the latter case, all leaves of $\mathcal{F}$ are dense. For a proof of this well-known fact, see \cite[Thm.~9.3.13]{conlon}.
	
	In the fibration case, the foliated cohomology $H^{\bullet }(\mathcal{F})$ is well-behaved, as we explain below. In the case with dense leaves, the foliated cohomology can be quite pathological. Below, we will focus on the case in which the period group is discrete, i.e. the setup is as follows:
	
	\vspace{0.2cm}
	
	\noindent\fbox{%
		\parbox{\textwidth}{%
			In the following, $M$ is a compact, connected manifold and $\theta\in\Omega^{1}(M)$ is a closed, nowhere-vanishing one-form with period group given by
			$
			\text{Per}(\theta)=c\mathbb{Z}
			$
			for some $c\in\mathbb{R}_{>0}$.
		}%
	}
	\vspace{0.2cm}

	Discreteness of $\text{Per}(\theta)$ imposes that $M$ is a mapping torus and that $\mathcal{F}$ is given by the fibers of the natural fibration $M\rightarrow S^{1}$. 
	To explain this, recall that given a manifold $L$ and a diffeomorphism $\varphi\in\text{Diff}(L)$, the mapping torus $\mathbb{T}_{\varphi}$ is defined to be the manifold
	\[
	\mathbb{T}_{\varphi}:=\frac{\mathbb{R}\times L}{(t,x)\sim(t+1,\varphi(x))}.
	\]
	\begin{lemma}
		Fix a leaf $L$ of $\mathcal{F}$. There exists a diffeomorphism $\varphi\in\text{Diff}(L)$ such that $M$ is diffeomorphic to the mapping torus $\mathbb{T}_{\varphi}$. Moreover, the diffeomorphism $M\cong\mathbb{T}_{\varphi}$ can be chosen to intertwine the foliation $\mathcal{F}$ on $M$ and the foliation on $\mathbb{T}_{\varphi}$ by fibers of $\mathbb{T}_{\varphi}\rightarrow S^{1}$.
	\end{lemma}
	\begin{proof}
		We only outline the proof. For more details, we refer to \cite[\S 9.3]{conlon} and \cite[\S 2.6.3]{thesis}.
		
		Fix a vector field $W\in\mathfrak{X}(M)$ such that $\theta(W)=c$. The flow $\phi_t$ of $W$ consists of automorphisms of the foliation $\mathcal{F}$, i.e. $\phi_t$ sends leaves to leaves. Since $\text{Per}(\theta)=c\mathbb{Z}$, we have for each leaf $L$ of $\mathcal{F}$ that
		\[
		\big\{t\in\mathbb{R}:\phi_t(L)=L\big\}=\mathbb{Z}.
		\]
		For a fixed leaf $L$, we now consider the smooth surjection
		\[
		\psi:\mathbb{R}\times L\rightarrow M: (t,p)\mapsto\phi_t(p)
		\]
		and note that
		\begin{align*}
			\phi_t(p)=\phi_{t'}(p')&\Leftrightarrow\begin{cases}t'=t+n\\ p=\phi_n(p')
			\end{cases}\ \ \text{for some}\ n\in\mathbb{Z}.
		\end{align*}
		This makes us define a $\mathbb{Z}$-action on $\mathbb{R}\times L$, given by
		\[
		n\cdot(t,p)=\left(t+n,(\phi_{-1})^{n}(p)\right).
		\]
		The action is free and properly discontinuous, and clearly the quotient manifold $(\mathbb{R}\times L)/\mathbb{Z}$ is precisely the mapping torus $\mathbb{T}_{\phi_{-1}}$. The map $\psi$ induces a diffeomorphism 
		\[
		\overline{\psi}:\mathbb{T}_{\phi_{-1}}\rightarrow M:[(t,p)]\mapsto\phi_t(p).
		\]
		Finally, it is clear that $\overline{\psi}$ takes the fibers of the natural fibration
		\[
		\mathbb{T}_{\phi_{-1}}\rightarrow S^{1}:[(t,p)]\mapsto t\ \text{mod}\ \mathbb{Z}
		\]
		to the leaves $\phi_t(L)$ of $\mathcal{F}$. This finishes the proof of the lemma.
	\end{proof}
	
	Hence in the following, we may assume that the manifold $M$ is a mapping torus $\mathbb{T}_{\varphi}$. Since the kernel of $\theta$ coincides with the vertical distribution of $\mathbb{T}_{\varphi}\rightarrow S^{1}$, we must have that $\theta$ is a basic one-form. We will also denote the angle coordinate on $S^{1}$ by $t$. Therefore, we have
	\[
	\theta=f(t)dt\in\Omega^{1}(S^1).
	\]
	In de Rham cohomology, we have that $[\theta]=[cdt]$ in $H^{1}(\mathbb{T}_{\varphi})$. Because the Morse-Novikov cohomology groups $H_{\theta}^{\bullet}(\mathbb{T}_{\varphi})$ only depend on the de Rham cohomology class of $\theta$ by Prop.~\ref{prop:prelim}, we may assume that $\theta=cdt$. In conclusion, we have reduced the computation of the Morse-Novikov cohomology groups $H^{k}_{\theta}(M)$ to the following special case.
	
	\begin{pbl}
		Compute the Morse-Novikov cohomology groups $H^{k}_{cdt}(\mathbb{T}_{\varphi})$ for $c\neq 0$.
	\end{pbl}

	To use the exact sequence \eqref{eq:sh-ex} in Cor.~\ref{cor:sec}, we need a convenient description for the foliated cohomology groups $H^{\bullet}(\mathcal{F})$, where $\mathcal{F}$ is the foliation by fibers of $\pi:\mathbb{T}_{\varphi}\rightarrow S^{1}$. The facts we are about to recall can be found in \cite[\S~4.1]{marcut}.
	It is well-known that the $k$-th cohomology groups of the fibers of $\pi$ constitute a smooth vector bundle, which we denote by $\mathcal{H}^{k}\rightarrow S^{1}$. 
	
	\begin{defi}
	The vector bundle $\mathcal{H}^{k}\rightarrow S^{1}$ is defined by $\mathcal{H}^{k}_{q}:=H^{k}\left(\pi^{-1}(q)\right)$ for $q\in S^{1}$.
	\end{defi}
	
	The fact that $\mathcal{H}^{k}\rightarrow S^{1}$ is indeed a smooth vector bundle was proved in \cite[Lemma~4.6]{marcut}.
	It turns out that the foliated cohomology group $H^{k}(\mathcal{F})$ can be identified with the space of sections $\Gamma(\mathcal{H}^{k})$. The explicit isomorphism is stated in the following result \cite[Lemma~4.8]{marcut}. 
	
	\begin{lemma}
	Let $\mathcal{F}$ be the foliation by fibers of $\pi:\mathbb{T}_{\varphi}\rightarrow S^{1}$. We then have an isomorphism 
	\[
	H^{k}(\mathcal{F})\overset{\sim}{\longrightarrow}\Gamma(\mathcal{H}^{k}):[\beta]\mapsto \sigma_{\beta},
	\]
	where the section $\sigma_{\beta}$ is defined by setting $\sigma_{\beta}(q)=\left[\beta|_{\pi^{-1}(q)}\right]\in H^{k}\left(\pi^{-1}(q)\right)$ for all $q\in S^{1}$.
	\end{lemma}
	
	It is well-known that the vector bundle $\mathcal{H}^{k}\rightarrow S^{1}$ carries a canonical flat connection $\nabla$, called Gauss-Manin connection \cite[Def.~4.7]{marcut}. In a local trivialization
	$
	\mathcal{H}^{k}|_U\cong U\times H^{k}(L),
	$
	the flat sections of $\nabla$ are locally constant. An explicit way of describing $\nabla$ is the following. Let $\mathfrak{X}(\mathbb{T}_{\varphi})^{\mathcal{F}}$ be the Lie subalgebra of $\left(\mathfrak{X}(\mathbb{T}_{\varphi}),[\cdot,\cdot]\right)$ consisting of $\mathcal{F}$-projectable vector fields, 
	%\[
	%\mathfrak{X}(\mathbb{T}_{\varphi})^{\mathcal{F}}:=\left\{Y\in\mathfrak{X}(\mathbb{T}_{\varphi}): [Y,\Gamma(T\mathcal{F})]\subset\Gamma(T\mathcal{F})\right\},
	%\]
	as introduced in Def.~\ref{def:projectable}. Clearly, it fits in a short exact sequence of Lie algebras
	\begin{equation}\label{eq:short-ex}
		0\longrightarrow \Gamma(T\mathcal{F})\longrightarrow \mathfrak{X}(\mathbb{T}_{\varphi})^{\mathcal{F}}\longrightarrow\mathfrak{X}(S^{1})\longrightarrow 0.
	\end{equation}
	
	\begin{defi}
	The Gauss-Manin connection $\nabla$ on $\mathcal{H}^{k}\rightarrow S^{1}$ is defined by setting
	\begin{equation}\label{eq:Gauss-Manin}
	\nabla_{V}\sigma_{\beta}:=\sigma_{\pounds_{\widetilde{V}}\beta}.
	\end{equation}
	for $\sigma_{\beta}\in\Gamma(\mathcal{H}^{k})$ and $V\in\mathfrak{X}(S^{1})$. Here $\widetilde{V}\in\mathfrak{X}(\mathbb{T}_{\varphi})^{\mathcal{F}}$ is any lift of $V$ in the sequence \eqref{eq:short-ex}.
	\end{defi}
	
	The Lie derivative appearing above is well-defined because $\widetilde{V}$ is $\mathcal{F}$-projectable, see Rem.~\ref{rem:Lie-der}. The expression \eqref{eq:Gauss-Manin} can be obtained by combining \cite[Thm.~4.1~(d)]{marcut} and \cite[Lemma~4.8]{marcut}.
	
	\vspace{0.2cm}
	
	With these preliminaries in place, we can now revisit the short exact sequence \eqref{eq:sh-ex}. We will rephrase it in terms of data attached to the flat vector bundles $(\mathcal{H}^{k},\nabla)$. Since the zeroth Morse-Novikov cohomology group is well-known (see Prop.~\ref{prop:prelim} $ii)$), we list it separately.

	\begin{cor}\label{cor:nabla}
		The Morse-Novikov cohomology $H^{\bullet}_{cdt}(\mathbb{T}_{\varphi})$ for $c\neq 0$ is given by:
		\begin{enumerate}[i)]
			\item $H^{0}_{cdt}(\mathbb{T}_{\varphi})=0$,
			\item For $k\geq 1$, we have a short exact sequence
			\begin{equation*}
				0\rightarrow\frac{\Gamma(\mathcal{H}^{k-1})}{\left\{\nabla_{\partial_t}\sigma_{\beta}-c\sigma_{\beta}: \sigma_{\beta}\in\Gamma(\mathcal{H}^{k-1})\right\}} \rightarrow H^{k}_{cdt}(\mathbb{T}_{\varphi})\rightarrow \left\{\sigma_{\beta}\in\Gamma(\mathcal{H}^{k}): \nabla_{\partial_t}\sigma_{\beta}=c\sigma_{\beta}\right\}\rightarrow 0.
			\end{equation*}
		\end{enumerate}
	\end{cor}
	
	For computational purposes, it is useful to note that sections of $\mathcal{H}^{k}\rightarrow S^{1}$ can be identified with $\mathbb{Z}$-invariant sections of $H^{k}(L)\times\mathbb{R}\rightarrow\mathbb{R}$. To explain this, note that the vector bundle $H^{k}(L)\times\mathbb{R}\rightarrow\mathbb{R}$ is equivariant, in the sense that the group $\mathbb{Z}$ acts on $H^{k}(L)\times\mathbb{R}$ by vector bundle automorphisms. Indeed, we have a commutative diagram for each $n\in\mathbb{Z}$:
	\[
	\hspace{-3cm}
	\begin{tikzcd}[column sep=10em]
		&H^{k}(L)\times\mathbb{R} \arrow{r}{\left([\alpha],t\right)\mapsto\left([(\varphi^{*})^{-n}(\alpha)],t+n\right)}\arrow[d] &H^{k}(L)\times\mathbb{R}\arrow[d]\\
		&\mathbb{R}\arrow[r,"t\mapsto t+n"]&\mathbb{R}
	\end{tikzcd}.
	\]
	The action $\mathbb{Z}\curvearrowright H^{k}(L)\times\mathbb{R}$ is free and properly discontinuous, since it covers the action $\mathbb{Z}\curvearrowright\mathbb{R}$ by translations. It follows that the quotient manifold $\big(H^{k}(L)\times\mathbb{R}\big)/\mathbb{Z}$ is a vector bundle over $S^{1}$, which is precisely the bundle $\mathcal{H}^{k}\rightarrow S^{1}$. See for instance \cite[Ex.~4.30]{Bunke}.
	We can then identify $\Gamma(\mathcal{H}^{k})$ with the space of $\mathbb{Z}$-invariant sections
	\[
	\Gamma\big(H^{k}(L)\times\mathbb{R}\big)^{\mathbb{Z}}:=\left\{[\sigma(t)]\in\Gamma\big(H^{k}(L)\times\mathbb{R}\big):[\sigma(t)]=[\varphi^{*}(\sigma(t+1))]\ \ \forall t\in\mathbb{R}\right\}.
	\]

	We now proceed with the computation of the cohomology groups $H^{k}_{cdt}(\mathbb{T}_{\varphi})$, using Cor.~\ref{cor:nabla}. The following proposition takes care of the last term appearing in Cor.~\ref{cor:nabla} $ii)$.
	
	\begin{prop}\label{prop:ker}
		For any $k\in\mathbb{N}$, we have an isomorphism of vector spaces
		\[
		\left\{\sigma_{\beta}\in\Gamma(\mathcal{H}^{k}): \nabla_{\partial_t}\sigma_{\beta}=c\sigma_{\beta}\right\}\overset{\sim}{\rightarrow} \ker\left([\varphi^{*}-e^{-c}\mathrm{Id}]:H^{k}(L)\rightarrow H^{k}(L)\right):\sigma_{\beta}\mapsto\sigma_{\beta}(0).
		\]
	\end{prop}
	\begin{proof}
		We view sections of $\mathcal{H}^{k}\rightarrow S^{1}$ as $\mathbb{Z}$-invariant sections of $H^{k}(L)\times\mathbb{R}\rightarrow\mathbb{R}$. Fix a basis $[\beta_1],\ldots[\beta_n]$ of $H^{k}(L)$. Any section $\sigma_{\beta}\in\Gamma(\mathcal{H}^{k})$ can be written as
		$$
		\sigma_{\beta}=\sum_{i=1}^{n}f_{i}(t)[\beta_i]
		$$
		for some $f_i\in C^{\infty}(\mathbb{R})$. The condition $\nabla_{\partial_t}\sigma_{\beta}=c\sigma_{\beta}$ translates to $f_i'=cf_i$, hence $f_i(t)=K_{i}e^{ct}$ for some constants $K_i\in\mathbb{R}$. Imposing that $\sum_{i}f_i(t)[\beta_i]$ is $\mathbb{Z}$-invariant means that
		\[
		[\varphi^{*}]\left(\sum_{i=1}^{n}K_{i}e^{c(t+1)}[\beta_i]\right)=\sum_{i=1}^{n}K_{i}e^{ct}[\beta_i],
		\]
		which is equivalent with
		\[
		[\varphi^{*}-e^{-c}\mathrm{Id}]\left(\sum_{i=1}^{n}K_{i}[\beta_i]\right)=0.\qedhere
		\]
	\end{proof}

	To deal with the first term appearing in Cor.~\ref{cor:nabla} $ii)$, we need the following lemma.
	
	\begin{lemma}\label{lem:ode}
		Fix a constant $c\in\mathbb{R}$ and a function $g\in C^{\infty}(\mathbb{R})$. The solution $f\in C^{\infty}(\mathbb{R})$ to the differential equation
		\[
		g(t)=f'(t)-cf(t)
		\]
		is given by
		\[
		f(t)=e^{ct}\left(\int_{0}^{t}g(s)e^{-cs}ds+K\right),\hspace{1cm}K\in\mathbb{R}.
		\]
	\end{lemma}
	\begin{proof}
		Clearly, the solution $f\in C^{\infty}(\mathbb{R})$ to the homogeneous equation $0=f'-cf$ is given by $f(t)=Ke^{ct}$ for $K\in\mathbb{R}$. Hence, the statement follows if we show that
		\[
		f(t)=e^{ct}\left(\int_{0}^{t}g(s)e^{-cs}ds\right)
		\]
		is a particular solution to the ODE. This immediately follows from the computation
		\[
		f'(t)-cf(t)=ce^{ct}\left(\int_{0}^{t}g(s)e^{-cs}ds\right)+e^{ct}g(t)e^{-ct}-ce^{ct}\left(\int_{0}^{t}g(s)e^{-cs}ds\right)=g(t).\qedhere
		\] 
	\end{proof}

	We can now study the first term in Cor.~\ref{cor:nabla} $ii)$. Again, we identify sections of $\mathcal{H}^{k}\rightarrow S^{1}$ with $\mathbb{Z}$-invariant sections of $H^{k}(L)\times\mathbb{R}\rightarrow\mathbb{R}$.
	
	\begin{prop}\label{prop:coker}
		\begin{enumerate}[i)]
			\item For every $k\in\mathbb{N}$ and $c\in\mathbb{R}$, we have a surjective linear map
			\begin{equation}\label{eq:surjection}
				\Gamma(\mathcal{H}^{k})\rightarrow H^{k}(L):\sigma_{\beta}\mapsto\int_{0}^{1}e^{-ct}\sigma_{\beta}(t)dt.
			\end{equation}
			\item For every $k\in\mathbb{N}$ and $c\in\mathbb{R}$, the map \eqref{eq:surjection} induces an isomorphism of vector spaces
			\[
			\frac{\Gamma(\mathcal{H}^{k})}{\left\{\nabla_{\partial_t}\sigma_{\beta}-c\sigma_{\beta}: \sigma_{\beta}\in\Gamma(\mathcal{H}^{k})\right\}}\overset{\sim}{\rightarrow} \mathrm{coker}\left([\varphi^{*}-e^{-c}\mathrm{Id}]:H^{k}(L)\rightarrow H^{k}(L)\right).
			\]
		\end{enumerate}
	\end{prop}
	\begin{proof}
		To prove item $i)$, we pick $[\beta]\in H^{k}(L)$. Take a bump function $\psi(t)\in C^{\infty}[0,1]$ such that $\psi\equiv 0$ near $0$ and $\psi\equiv 0$ near $1$. Then $\psi(t)[\beta]\in\Gamma(H^{k}(L)\times[0,1])$ can be extended to a smooth $\mathbb{Z}$-invariant section of $H^{k}(L)\times\mathbb{R}$, i.e. a section $\sigma_{\gamma}$ of $\mathcal{H}^{k}$. We get that the image of 
		\[
		\left(\int_{0}^{1}e^{-ct}\psi(t)dt\right)^{-1}\sigma_{\gamma}\in\Gamma(\mathcal{H}^{k})
		\]
		under the map \eqref{eq:surjection} is exactly $[\beta]\in H^{k}(L)$. This shows that the map \eqref{eq:surjection} is surjective.
		
		To prove item $ii)$, we pick a section $\sigma_{\gamma}\in\Gamma(\mathcal{H}^{k})$ and find the precise conditions under which $\sigma_{\gamma}=\nabla_{\partial_t}\sigma_{\beta}-c\sigma_{\beta}$ for some $\sigma_{\beta}\in\Gamma(\mathcal{H}^{k})$. Fixing a basis $[\beta_1],\ldots,[\beta_n]$ of $H^{k}(L)$, we can write
		$
		\sigma_{\gamma}=\sum_{i}g_i(t)[\beta_i]
		$
		for some $g_i\in C^{\infty}(\mathbb{R})$, and we have the $\mathbb{Z}$-invariance condition
		\begin{equation}\label{eq:Z-invariance}
			[\varphi^{*}]\left(\sum_{i=1}^{n}g_i(t+1)[\beta_i]\right)=\sum_{i=1}^{n}g_i(t)[\beta_i].
		\end{equation}
		For there to exist $\sigma_{\beta}=\sum_{i}f_i(t)[\beta_i]$ such that $\sigma_{\gamma}=\nabla_{\partial_t}\sigma_{\beta}-c\sigma_{\beta}$, necessarily $g_i=f_{i}'-cf_i$. By Lemma~\ref{lem:ode}, we then have
		\[
		f_i(t)=e^{ct}\left(\int_{0}^{t}g_i(s)e^{-cs}ds+K_i\right)
		\]
		for some $K_i\in\mathbb{R}$. For $\sigma_{\beta}=\sum_{i}f_i(t)[\beta_i]$ to be a well-defined section of $\mathcal{H}^{k}$, it moreover needs to satisfy the $\mathbb{Z}$-invariance condition
		\begin{equation}\label{eq:invariance}
			[\varphi^{*}]\left(\sum_{i=1}^{n}e^{c(t+1)}\left(\int_{0}^{t+1}g_i(s)e^{-cs}ds+K_i\right)[\beta_i]\right)=\sum_{i=1}^{n}e^{ct}\left(\int_{0}^{t}g_i(s)e^{-cs}ds+K_i\right)[\beta_i].
		\end{equation}
		We now inspect the left hand side of \eqref{eq:invariance}. Note that
		\begin{align*}
			\int_{0}^{t+1}g_i(s)e^{-cs}ds&=\int_{0}^{1}g_i(s)e^{-cs}ds+\int_{1}^{t+1}g_i(s)e^{-cs}ds\\
			&=\int_{0}^{1}g_i(s)e^{-cs}ds+\int_{0}^{t}g_{i}(s+1)e^{-cs-c}ds.
		\end{align*}
		Invoking the $\mathbb{Z}$-invariance \eqref{eq:Z-invariance}, we see that the left hand side of \eqref{eq:invariance} is given by
		\begin{align*}
			&[\varphi^{*}]\left(\sum_{i=1}^{n}e^{c(t+1)}\left(\int_{0}^{1}g_i(s)e^{-cs}ds+K_i\right)[\beta_i]\right)+\sum_{i=1}^{n}e^{c(t+1)}\left(\int_{0}^{t}g_{i}(s)e^{-cs-c}ds\right)[\beta_i]\\
			&=[\varphi^{*}]\left(\sum_{i=1}^{n}e^{c(t+1)}\left(\int_{0}^{1}g_i(s)e^{-cs}ds+K_i\right)[\beta_i]\right)+\sum_{i=1}^{n}e^{ct}\left(\int_{0}^{t}g_{i}(s)e^{-cs}ds\right)[\beta_i].
		\end{align*}
		Hence, the requirement \eqref{eq:invariance} is equivalent with
		\[
		[\varphi^{*}]\left(\sum_{i=1}^{n}\left(\int_{0}^{1}g_i(s)e^{-cs}ds+K_i\right)[\beta_i]\right)=e^{-c}\sum_{i=1}^{n}K_{i}[\beta_i],
		\]
		or put differently,
		\begin{align*}
			[\varphi^{*}-e^{-c}\mathrm{Id}]\left(\sum_{i=1}^{n}K_{i}[\beta_i]\right)&=-[\varphi^{*}]\left(\sum_{i=1}^{n}\left(\int_{0}^{1}g_i(s)e^{-cs}ds\right)[\beta_i]\right)\\
			&=-[\varphi^{*}]\left(\int_{0}^{1}e^{-cs}\sigma_{\gamma}(s)ds\right).
		\end{align*}
		It follows that $\sigma_{\gamma}=\nabla_{\partial_t}\sigma_{\beta}-c\sigma_{\beta}$ for some $\sigma_{\beta}\in\Gamma(\mathcal{H}^{k})$ exactly when
		\[
		-[\varphi^{*}]\left(\int_{0}^{1}e^{-cs}\sigma_{\gamma}(s)ds\right)\in\mathrm{Im}\left([\varphi^{*}-e^{-c}\mathrm{Id}]:H^{k}(L)\rightarrow H^{k}(L)\right),
		\]
		which in turn is equivalent with
		\[
		\int_{0}^{1}e^{-cs}\sigma_{\gamma}(s)ds\in\mathrm{Im}\left([\varphi^{*}-e^{-c}\mathrm{Id}]:H^{k}(L)\rightarrow H^{k}(L)\right).
		\]
		This shows that the linear surjection
		\[
		\Gamma(\mathcal{H}^{k})\rightarrow H^{k}(L):\sigma_{\gamma}\mapsto\int_{0}^{1}e^{-cs}\sigma_{\gamma}(s)ds.
		\]
		considered in part $i)$ of the proposition descends to an isomorphism
		\[
		\frac{\Gamma(\mathcal{H}^{k})}{\left\{\nabla_{\partial_t}\sigma_{\beta}-c\sigma_{\beta}: \sigma_{\beta}\in\Gamma(\mathcal{H}^{k})\right\}}\overset{\sim}{\longrightarrow} \mathrm{coker}\left([\varphi^{*}-e^{-c}\mathrm{Id}]:H^{k}(L)\rightarrow H^{k}(L)\right).\qedhere
		\]
	\end{proof}
	
	Combining Cor.~\ref{cor:nabla} with Prop.~\ref{prop:ker} and Prop.~\ref{prop:coker}, we obtain the following. We denote
	\begin{align*}
	&K^{\bullet}_c:=\ker\left([\varphi^{*}-e^{-c}\mathrm{Id}]:H^{\bullet}(L)\rightarrow H^{\bullet}(L)\right),\\
	&C^{\bullet}_c:=\mathrm{coker}\left([\varphi^{*}-e^{-c}\mathrm{Id}]:H^{\bullet}(L)\rightarrow H^{\bullet}(L)\right).
	\end{align*}
	\begin{cor}\label{cor:computed}
		The Morse-Novikov cohomology $H^{\bullet}_{cdt}(\mathbb{T}_{\varphi})$ for $c\neq 0$ is given as follows.
		\begin{enumerate}[i)]
			\item $H^{0}_{cdt}(\mathbb{T}_{\varphi})=0$,
			\item For $k\geq 1$, we have a short exact sequence
			\begin{align}\label{eq:seq-C-K}
				0\longrightarrow C_{c}^{k-1}\longrightarrow H^{k}_{cdt}(\mathbb{T}_{\varphi})\longrightarrow K_{c}^{k}\longrightarrow 0.
			\end{align}
		\end{enumerate}
	\end{cor}

	%\begin{remark}\label{rem:splitting}
	%\textcolor{red}{Explicit expression for splitting}
	%\end{remark}
	
	\begin{remark}
	Cor.~\ref{cor:computed} is related with various results in the literature that were obtained via other methods, like Mayer-Vietoris techniques or Novikov homology. The novelty of our result lies in the fact that we use a geometric approach, relying on foliation theory instead.
		\begin{enumerate}
			\item Moroianu and Pilca \cite{Moroianu} have showed the following. Viewing a matrix $A\in\mathrm{SL}_{2}(\mathbb{Z})$ as a diffeomorphism of the torus $\mathbb{T}^{2}$, one can consider the mapping torus $\mathbb{T}_{A}\rightarrow S^{1}$. It is shown in \cite[Lemma~4.1]{Moroianu} that the first Morse-Novikov cohomology $H^{1}_{cdt}(\mathbb{T}_{A})$ for $c\in\mathbb{R}_{0}$ is non-zero exactly when $e^{c}$ is an eigenvalue of $A$. This is done by relating $H^{1}_{cdt}(\mathbb{T}_{A})$ to certain $2$-dimensional representations of the fundamental group $\pi_1(\mathbb{T}_{A})$.
			We can recover this result using Cor.~\ref{cor:computed} as follows. Since $c\neq 0$, we have
			\[
			\mathrm{coker}\left(\mathrm{Id}-e^{-c}\mathrm{Id}:H^{0}(\mathbb{T}^{2})\rightarrow H^{0}(\mathbb{T}^{2})\right)=\frac{\mathbb{R}}{(1-e^{-c})\cdot\mathbb{R}}=0.
			\]
			Hence by Cor.~\ref{cor:computed}, we have that
			\[
			H^{1}_{cdt}(\mathbb{T}_{A})\cong\ker\left(A^{T}-e^{-c}\mathrm{Id}:H^{1}(\mathbb{T}^{2})\rightarrow H^{1}(\mathbb{T}^{2})\right).
			\]
			In particular, $H^{1}_{cdt}(\mathbb{T}_{A})$ is non-zero exactly when $e^{-c}$ is an eigenvalue of $A^{T}$. Since $A$ and $A^{T}$ have the same eigenvalues, this is equivalent with $e^{-c}$ being an eigenvalue of $A$. Since $\det A=1$, this in turn is equivalent with $e^{c}$ being an eigenvalue of $A$.
			\item Pajitnov \cite{Pajitnov} showed that all Morse-Novikov cohomology groups $H^{k}_{cdt}(\mathbb{T}_{\varphi})$ vanish when $e^{-c}$ is transcendental. This is done by proving that the dimension of $ H^{k}_{cdt}(\mathbb{T}_{\varphi})$ agrees with the $k$-th Novikov Betti number \cite{Novikov} of the closed one-form $cdt$, which is zero whenever the closed one-form in question is nowhere-vanishing.
			
			This result is consistent with Cor.~\ref{cor:computed}. To see this, it suffices to argue that a transcendental number $e^{-c}$ cannot be an eigenvalue of $[\varphi^{*}]:H^{k}(L)\rightarrow H^{k}(L)$ for any $k\in\mathbb{N}$. It then follows that the maps $[\varphi^{*}-e^{-c}\mathrm{Id}]:H^{k}(L)\rightarrow H^{k}(L)$ are isomorphisms, hence their kernels and cokernels vanish. To show that the eigenvalues of the map $[\varphi^{*}]:H^{k}(L)\rightarrow H^{k}(L)$ are algebraic numbers, it suffices to show the same for the map in homology with real coefficients $[\varphi_{*}]_{\mathbb{R}}:H_{k}(L;\mathbb{R})\rightarrow H_{k}(L;\mathbb{R})$. Indeed, with respect to dual bases for $H_{k}(L;\mathbb{R})$ and $H^{k}(L)=H_{k}(L;\mathbb{R})^{*}$, the matrices of $[\varphi^{*}]$ and $[\varphi_{*}]_{\mathbb{R}}$ are transposes of each other hence they share the same eigenvalues. But the fact that the eigenvalues of $[\varphi_{*}]_{\mathbb{R}}:H_{k}(L;\mathbb{R})\rightarrow H_{k}(L;\mathbb{R})$ are algebraic numbers follows immediately from naturality of the universal coefficient theorem
			\[
			\begin{tikzcd}
				&H_{k}(L;\mathbb{Z})\otimes\mathbb{R}\arrow[r,"\sim"]\arrow{d}{[\varphi_{*}]_{\mathbb{Z}}\otimes\mathrm{Id}}&H_{k}(L;\mathbb{R})\arrow{d}{[\varphi_{*}]_{\mathbb{R}}}\\
				&H_{k}(L;\mathbb{Z})\otimes\mathbb{R}\arrow[r,"\sim"]&H_{k}(L;\mathbb{R})
			\end{tikzcd}.
			\]
			This implies that the matrix of $[\varphi_{*}]_{\mathbb{R}}:H_{k}(L;\mathbb{R})\rightarrow H_{k}(L;\mathbb{R})$, with respect to a basis coming from generators for the free part of $H_{k}(L;\mathbb{Z})$, has integer coefficients. Consequently, the roots of its characteristic polynomial are algebraic numbers.

			Pajitnov \cite{Pajitnov} also gives an expression for the dimension of $H^{k}_{cdt}(\mathbb{T}_{\varphi})$ when $e^{-c}$ is algebraic, which is however not explicit. He describes the dimension in terms of certain unspecified polynomials appearing in a decomposition of a rational cohomology group viewed as a finitely generated module over the principal ideal ring $\mathbb{Q}[z,z^{-1}]$. Cor.~\ref{cor:computed} has the advantage of being explicit, simple and geometric.
			
			\item Otiman \cite{Otiman} has computed the Morse-Novikov cohomology groups $H^{\bullet}_{\theta}(\mathcal{S})$ of Inoue surfaces $\mathcal{S}$, with respect to a closed one-form $\theta$ which is the Lee form for a certain natural locally conformal symplectic structure. Since the manifolds $\mathcal{S}$ in question are mapping tori $\mathcal{S}\rightarrow S^{1}$ and the Lee form $\theta$ is basic, this could also be done using Cor.~\ref{cor:computed}. The original argument in  \cite{Otiman} uses the Mayer-Vietoris sequence for Morse-Novikov cohomology instead.  
		\end{enumerate}
	\end{remark}
	
	\begin{ex}
		Assume that $L$ is orientable and that there is a Riemannian metric $g$ on $L$ such that the diffeomorphism $\varphi$ belongs to the isometry group $\text{Isom}(L,g)$. In that case, all Morse-Novikov cohomology groups $H^{\bullet}_{cdt}(\mathbb{T}_{\varphi})$ for $c\neq 0$ vanish because the maps
		\[
		[\varphi^{*}-e^{-c}\mathrm{Id}]:H^{k}(L)\rightarrow H^{k}(L)
		\]
		are isomorphisms for all $k\geq 0$. Indeed, assume that $[\alpha]\in H^{k}(L)$ would be a non-zero class in the kernel of $[\varphi^{*}-e^{-c}\mathrm{Id}]$. Let $\alpha\in\Omega^{k}(L)$ be the unique harmonic representative of $[\alpha]$. 
		Since $\varphi$ is an isometry, it follows that $\varphi^{*}\alpha$ is the unique harmonic representative of the class $[e^{-c}\alpha]$, hence $\varphi^{*}\alpha=e^{-c}\alpha$. Denoting by $\langle\cdot,\cdot\rangle_{g}$ the $L^{2}$-inner product on $\Omega^{k}(L)$, we get
		\[
		\langle\alpha,\alpha\rangle_g=\langle\varphi^{*}\alpha,\varphi^{*}\alpha\rangle_g=e^{-2c}\langle\alpha,\alpha\rangle_g.
		\]
		Since $\alpha$ is non-zero, this would force that $c=0$, hence we reach a contradiction.
	\end{ex}

	\section{Rigidity of Lie affine foliations}\label{sec:deformations}
	
	This section is devoted to rigidity of Lie affine foliations. Our motivation is the paper \cite{Lie} by El Kacimi-Guasp-Nicolau, where the authors study deformations of a given Lie foliation within the class of Lie foliations. They show that a Lie foliation on a compact manifold is rigid, provided that a certain cohomology group $H^{1}(\mathcal{A})$ vanishes. The aim of our paper is to complement their result by looking for situations in which this rigidity theorem applies. Since the authors of \cite{Lie} studied the cohomology  $H^{1}(\mathcal{A})$ in detail for Lie foliations modeled on abelian Lie algebras, we will focus on Lie foliations modeled on the simplest non-abelian Lie algebra $\mathfrak{aff}(1)$. We refer to them as \emph{Lie affine foliations}. We study the cohomology group $H^{1}(\mathcal{A})$ in this situation, showing that it can be expressed in terms of Morse-Novikov cohomology. We give precise criteria for its vanishing, which leads to many examples of rigid Lie affine foliations. We are particularly interested in a suitable class of Lie affine foliations on mapping tori which we call \emph{model Lie affine foliations}. Specializing our rigidity result to this class, we obtain in particular that every Lie affine foliation on a compact, connected, oriented manifold of dimension $3$ or $4$ is rigid among nearby Lie foliations.

	\subsection{Deformations of Lie foliations}\label{subsec:def-lie}
	In this subsection, we recall some background about deformations of Lie foliations. The material is taken from \cite{Lie}, where one actually studies deformations of a class of transversely homogeneous foliations that includes Lie foliations.
	
	\begin{defi}
		Let $M$ be a manifold and $\mathfrak{g}$ a Lie algebra. A Maurer-Cartan form with values in $\mathfrak{g}$ is a differential form $\omega\in\Omega^{1}(M,\mathfrak{g})$ satisfying the Maurer-Cartan equation
		\[
		d\omega+\frac{1}{2}[\omega,\omega]=0.
		\]
		Assume moreover that $\omega$ is non-singular, i.e. $\omega_x:T_{x}M\rightarrow\mathfrak{g}$ is surjective for all $x\in M$. Then  $\ker\omega$ integrates to a foliation $\mathcal{F}$, which we call the Lie $\mathfrak{g}$-foliation defined by $\omega\in\Omega^{1}(M,\mathfrak{g})$.
	\end{defi}
	
	From now on, we fix a basis $\{e_1,\ldots,e_m\}$ for the Lie algebra $\mathfrak{g}$ and denote the associated structure constants by $K_{ij}^{k}$. The Lie $\mathfrak{g}$-foliation $\mathcal{F}$ is given by
	one-forms $\omega_1,\ldots,\omega_m\in\Omega^{1}(M)$ that are linearly independent at every point and satisfy
	\begin{equation}\label{eq:lie}
		d\omega_k=-\frac{1}{2}\sum_{i,j=1}^{m}K_{ij}^{k}\omega_i\wedge\omega_j.
	\end{equation}
	
	\noindent\fbox{%
		\parbox{\textwidth}{%
			For us, a Lie foliation will be a pair $(\mathcal{F},\omega)$ where $\omega\in\Omega^{1}(M,\mathfrak{g})$ is a non-singular Maurer-Cartan form and $\mathcal{F}$ is the foliation defined by $\omega$. We emphasize that the Maurer-Cartan form will be considered as part of the data of a Lie foliation.
		}%
	}
	\vspace{0.2cm}
	
	We now specify what we mean by deforming $(\mathcal{F},\omega)$ as a Lie foliation. 
	
	\begin{defi}\label{def:Lie-fol}
		Let $(\mathcal{F},\omega)$ be a Lie $\mathfrak{g}$-foliation defined by $\omega=(\omega_1,\ldots,\omega_m)\in\Omega^{1}(M,\mathfrak{g})$.
		\begin{enumerate}[i)]
			\item A smooth deformation $(\mathcal{F}_t,\omega_t)$ of $(\mathcal{F},\omega)$ parametrized by an interval $I\subset\mathbb{R}$ around $0$ is determined by a smoothly varying family of one-forms $\omega_1(t),\ldots,\omega_m(t)$ that are linearly independent at every point and a set of smooth functions $K_{ij}^{k}(t)$ such that
			\begin{equation*}
				d\omega_k(t)=-\frac{1}{2}\sum_{i,j=1}^{m}K_{ij}^{k}(t)\omega_i(t)\wedge\omega_j(t),
			\end{equation*}
			and $K_{ij}^{k}(t)$ are the structure constants of a Lie algebra $\mathfrak{g}_t$. We also require $\omega_j(0)=\omega_j$.
			\item The Lie foliation $(\mathcal{F},\omega)$ is rigid if every smooth deformation $(\mathcal{F}_t,\omega_t)$ of $(\mathcal{F},\omega)$ is trivial. That is, shrinking the interval $I$ if necessary, there is a smooth family of diffeomorphisms $\phi_t:M\rightarrow M$ and a family of Lie algebra isomorphisms $\nu_t:\mathfrak{g}\rightarrow\mathfrak{g}_t$ such that $$\phi_t^{*}\omega_t=\nu_t\circ\omega.$$
			We require moreover that $\phi_0:M\rightarrow M$ and $\nu_0:\mathfrak{g}\rightarrow\mathfrak{g}$ are the identity maps.
		\end{enumerate} 
	\end{defi}

	\begin{remark}
	We emphasize that we will be deforming the Maurer-Cartan form $\omega\in\Omega^{1}(M,\mathfrak{g})$, rather than just its associated foliation $\mathcal{F}$. In particular, the rigidity problem of the Lie $\mathfrak{g}$-foliation $(\mathcal{F},\omega)$ is very different from the rigidity problem of the underlying foliation $\mathcal{F}$. This will be illustrated in \S\ref{subsec:model} where we construct Lie $\mathfrak{aff}(1)$-foliations $(\mathcal{F},\omega)$ that are rigid when deformed as Lie foliations, whereas $\mathcal{F}$ is not rigid when deformed merely as a foliation.
	\end{remark}

	The deformation problem of a Lie foliation $(\mathcal{F},\omega)$ is governed at the infinitesimal level by a suitable cochain complex $\big(\mathcal{A}^{\bullet},D\big)$ which we now recall. Below, we denote by $\{\theta_1,\ldots,\theta_m\}$ the basis of $\mathfrak{g}^{*}$ dual to the basis $\{e_1,\ldots,e_m\}$ of $\mathfrak{g}$.
	\begin{itemize}
		\item We introduce an operator $\widehat{d}_{M}:\Omega^{r}(M)^{\times m}\rightarrow\Omega^{r+1}(M)^{\times m}$ defined as follows. Given $\sigma=(\sigma_1,\ldots,\sigma_m)\in\Omega^{r}(M)^{\times m}$, the $k$-th component of $\widehat{d}_{M}\sigma\in\Omega^{r+1}(M)^{\times m}$ is
		\[
		\big(\widehat{d}_{M}\sigma\big)_{k}:=d\sigma_k+\sum_{i,j=1}^{m}K_{ij}^{k}\omega_i\wedge\sigma_j.
		\]
		\item We also define an operator $\widehat{d}_{\mathfrak{g}}:\wedge^{r}\mathfrak{g}^{*}\otimes\mathfrak{g}\rightarrow\wedge^{r+1}\mathfrak{g}^{*}\otimes\mathfrak{g}$ as follows. If $\psi=\sum_{k=1}^{m}\psi_k\otimes e_k$ for $\psi_k\in\wedge^{r}\mathfrak{g}^{*}$, then we set
		\[
		\widehat{d}_{\mathfrak{g}}\psi:=\sum_{k=1}^{m}\Big(d\psi_k+\sum_{i,j=1}^{m}K_{ij}^{k}\theta_i\wedge\psi_j\Big)\otimes e_k.
		\]
		Here $d$ is the exterior derivative on the simply connected Lie group $G$ integrating $\mathfrak{g}$.
		\item At last, we define a linear map $\mathcal{R}:\wedge^{r}\mathfrak{g}^{*}\otimes\mathfrak{g}\rightarrow\Omega^{r}(M)^{\times m}$ by setting
		\[
		\left(\mathcal{R}(\theta_{j_1}\wedge\cdots\wedge\theta_{j_r}\otimes e_i)\right)_{k}:=\delta_{ik}\omega_{j_1}\wedge\cdots\wedge\omega_{j_r}.
		\]
	\end{itemize}
	
	\begin{defi}\label{def:A}
		For all $r\in\mathbb{N}$, we set $$\mathcal{A}^{r}:=\Omega^{r}(M)^{\times m}\oplus\big(\wedge^{r+1}\mathfrak{g}^{*}\otimes\mathfrak{g}\big).$$
		We get a differential $D:\mathcal{A}^{r}\rightarrow\mathcal{A}^{r+1}$, defined by
		\[
		D(\sigma,\psi)=\big(\widehat{d}_{M}\sigma-\mathcal{R}(\psi),-\widehat{d}_{\mathfrak{g}}\psi\big)
		\]
		for $\sigma\in\Omega^{r}(M)^{\times m}$ and $\psi\in\wedge^{r+1}\mathfrak{g}^{*}\otimes\mathfrak{g}$. 
	\end{defi}
	
	Given the Lie $\mathfrak{g}$-foliation $(\mathcal{F},\omega)$ defined by $\omega_1,\ldots,\omega_m\in\Omega^{1}(M)$ satisfying \eqref{eq:lie}, a smooth deformation $(\mathcal{F}_t,\omega_t)$ of $(\mathcal{F},\omega)$ is specified by two pieces of data:
	\begin{enumerate}[i)]
		\item A smooth path $\sigma(t)=\big(\sigma_1(t),\ldots,\sigma_m(t)\big)$ in $\Omega^{1}(M)^{\times m}$ with $\sigma(0)=0$,
		\item A smooth path $\psi(t)$ in $\wedge^{2}\mathfrak{g}^{*}\otimes\mathfrak{g}$ given by
		\[
		\psi(t)=-\frac{1}{2}\sum_{i,j,k=1}^{m}C_{ij}^{k}(t)\theta_i\wedge\theta_j\otimes e_k\hspace{1cm}\text{with}\ C_{ij}^{k}(0)=0.
		\]
	\end{enumerate}
	We then require that
	\begin{equation}\label{eq:1}
		d\big(\omega_k+\sigma_k(t)\big)+\frac{1}{2}\sum_{i,j=1}^{m}\left(K_{ij}^{k}+C_{ij}^{k}(t)\right)\big(\omega_i+\sigma_i(t)\big)\wedge\big(\omega_j+\sigma_j(t)\big)=0.
	\end{equation}
	We also need that $L_{ij}^{k}(t):=K_{ij}^{k}+C_{ij}^{k}(t)$ are the structure constants of a Lie algebra, i.e. the Jacobi identity needs to be satisfied. This amounts to
	\begin{align}\label{eq:2}
		&\sum_{i=1}^{m}\big(L_{ij}^{k}(t)L_{rs}^{i}(t)\big)+L_{ir}^{k}(t)L_{sj}^{i}(t)+L_{is}^{k}(t)L_{jr}^{i}(t)\big)=0.
	\end{align}
	The infinitesimal deformation coming from the path $\big(\sigma(t),\psi(t)\big)$ satisfies the linearization of the equations \eqref{eq:1} and \eqref{eq:2}, and this turns out to be
	\[
	D\big(\dt{\sigma(0)},\dt{\psi}(0)\big)=0.
	\]
	Moreover, if the path $\big(\sigma(t),\psi(t)\big)$ defines a trivial deformation of $(\mathcal{F},\omega)$ as in $ii)$ of Def.~\ref{def:Lie-fol}, then the associated infinitesimal deformation $\big(\dt{\sigma(0)},\dt{\psi}(0)\big)$ is exact in $\big(\mathcal{A}^{\bullet},D\big)$. Hence, vanishing of the first cohomology group $H^{1}(\mathcal{A})$ can be interpreted heuristically as the infinitesimal requirement for rigidity of the Lie $\mathfrak{g}$-foliation $(\mathcal{F},\omega)$. When the manifold $M$ is compact, this can be turned into an actual rigidity statement \cite[Thm.~2]{Lie}.
	
	\begin{thm}\label{thm:rigidity}
		Let $M$ be a compact manifold and $(\mathcal{F},\omega)$ a Lie $\mathfrak{g}$-foliation on $M$. If $H^{1}(\mathcal{A})$ vanishes, then $(\mathcal{F},\omega)$ is rigid.
	\end{thm}
	
	The aim of this paper is to find new examples of rigid Lie foliations using Thm.~\ref{thm:rigidity}. Since the cohomology $H^{1}(\mathcal{A})$ has been studied when the Lie algebra $\mathfrak{g}$ is abelian \cite[\S 6.4]{Lie}, we will look at Lie foliations modeled on the simplest non-abelian Lie algebra $\mathfrak{g}=\mathfrak{aff}(1)$. It turns out that their deformation cohomology $H^{1}(\mathcal{A})$ can be expressed in terms of Morse-Novikov cohomology, which is where the results of \S\ref{sec:one} come into play.

	\subsection{Deformations of Lie affine foliations}
	
	From now on, $(\mathcal{F},\omega)$ is a Lie $\mathfrak{aff}(1)$-foliation on $M$. In this subsection, we will specialize Thm.~\ref{thm:rigidity} to the case in which $\mathfrak{g}=\mathfrak{aff}(1)$. 
	
	Fix a basis $\{e_1,e_2\}$ of $\mathfrak{aff}(1)$ such that $[e_1,e_2]=e_1$, and let $\{\theta_1,\theta_2\}$ be the dual basis of $\mathfrak{aff}(1)^{*}$. The only non-zero structure constants $K_{ij}^{k}$ for $1\leq i,j,k\leq 2$ are $K_{12}^{1}=1$ and $K_{21}^{1}=-1$. Hence, a Lie $\mathfrak{aff}(1)$-foliation is defined by independent one-forms $\omega_1,\omega_2\in\Omega^{1}(M)$ satisfying the equations
	\begin{equation*}
		\begin{cases}
			d\omega_1-\omega_2\wedge\omega_1=0\\
			d\omega_2=0
		\end{cases}.
	\end{equation*}
	We now describe the infinitesimal deformations of the Lie $\mathfrak{aff}(1)$-foliation $(\mathcal{F},\omega)$.
	
	\begin{prop}\label{prop:coho}
		\begin{enumerate}
			\item One-cocycles in the complex $\big(\mathcal{A}^{\bullet},D\big)$ are given by quadruples $(\sigma_1,\sigma_2,c_1,c_2)\in\Omega^{1}(M)^{\times 2}\times\mathbb{R}^{2}$ such that
			\begin{equation}\label{eq:to-obtain}
				\begin{cases}
					d\left(\sigma_2-\frac{1}{2}c_2\omega_1+\frac{1}{2}c_1\omega_2\right)=0\\
					\left(\sigma_2-\frac{1}{2}c_2\omega_1+\frac{1}{2}c_1\omega_2\right)\wedge\omega_1=d\sigma_1-\omega_2\wedge\sigma_1
				\end{cases}.
			\end{equation}
			\item A one-cocycle $(\sigma_1,\sigma_2,c_1,c_2)$ is exact iff. there exist $f,g\in C^{\infty}(M)$ such that
			\begin{equation}\label{eq:coboundaries}
				\begin{cases}
					\sigma_2-\frac{1}{2}c_2\omega_1+\frac{1}{2}c_1\omega_2=df\\
					\sigma_1=f\omega_1-dg+g\omega_2
				\end{cases}.
			\end{equation}
		\end{enumerate}
	\end{prop}
	\begin{proof}
		$(1)$ Recall that elements $(\sigma,\psi)\in\mathcal{A}^{1}$ are given by $\sigma=(\sigma_1,\sigma_2)\in\Omega^{1}(M)^{\times 2}$ and 
		\[
		\psi=-\frac{1}{2}\big(c_1\theta_1\wedge\theta_2\otimes e_1+c_2\theta_1\wedge\theta_2\otimes e_2\big),\hspace{1cm}c_1,c_2\in\mathbb{R}.
		\]
		Asking that $D(\sigma,\psi)=0$ is equivalent with $\widehat{d}_{M}\sigma-\mathcal{R}\psi=0$. Indeed, we automatically have that $\widehat{d}_{\mathfrak{g}}\psi=0$ because $\widehat{d}_{\mathfrak{g}}\psi\in\wedge^{3}\mathfrak{g}^{*}\otimes\mathfrak{g}$ and $\mathfrak{g}=\mathfrak{aff}(1)$ is $2$-dimensional. We obtain the system 
		\begin{equation}\label{eq:system}
			\begin{cases}
				0=\big(\widehat{d}_{M}\sigma-\mathcal{R}\psi\big)_1=d\sigma_1+\omega_1\wedge\sigma_2-\omega_2\wedge\sigma_1+\frac{1}{2}c_1\omega_1\wedge\omega_2\\
				0=\big(\widehat{d}_{M}\sigma-\mathcal{R}\psi\big)_2=d\sigma_2+\frac{1}{2}c_2\omega_1\wedge\omega_2
			\end{cases}.
		\end{equation}
		Clearly, the first equation in \eqref{eq:system} coincides with the second equation in \eqref{eq:to-obtain}. The second equation in \eqref{eq:system} is equivalent with the first equation in \eqref{eq:to-obtain} since $d\omega_1=\omega_2\wedge\omega_1$.

		For item $(2)$, recall that a one-cocycle $(\sigma,\psi)\in\mathcal{A}^{1}$ as in part $(1)$ is exact iff. there exist $h_1,h_2\in C^{\infty}(M)$ and $\nu\in\mathfrak{g}^{*}\otimes\mathfrak{g}$ such that
		\begin{equation}\label{eq:coboundary}
			\begin{cases}
				\sigma=\widehat{d}_{M}(h_1,h_2)-\mathcal{R}(\nu)\\
				\psi=-\widehat{d}_{\mathfrak{g}}\nu	
			\end{cases}.
		\end{equation}
		We can write $\nu=\sum_{i,j=1}^{2}a_{ij}\theta_i\otimes e_j$ for some $a_{ij}\in\mathbb{R}$. The first equation in \eqref{eq:coboundary} is then equivalent with the system
		\[
		\begin{cases}
			\sigma_1=dh_1+h_2\omega_1-h_1\omega_2-a_{11}\omega_1-a_{21}\omega_2\\
			\sigma_2=dh_2-a_{12}\omega_1-a_{22}\omega_2
		\end{cases}.
		\]
		The second equation in \eqref{eq:coboundary} means that
		\begin{align}\label{eq:comp}
			&-\frac{1}{2}\big(c_1\theta_1\wedge\theta_2\otimes e_1+c_2\theta_1\wedge\theta_2\otimes e_2\big)\nonumber\\
			&\hspace{2cm}=-\Big(a_{11}d\theta_1+a_{21}d\theta_2+\theta_1\wedge\big(a_{12}\theta_1+a_{22}\theta_2\big)-\theta_2\wedge\big(a_{11}\theta_1+a_{21}\theta_2\big)\Big)\otimes e_1\nonumber\\
			&\hspace{2.45cm}-\left(a_{12}d\theta_1+a_{22}d\theta_2\right)\otimes e_2.
		\end{align}
		Since $[e_1,e_2]=e_1$, we have that
		\[
		\begin{cases}
			d\theta_1(e_1,e_2)=-\theta_1\big([e_1,e_2]\big)=-\theta_1(e_1)=-1\\
			d\theta_2(e_1,e_2)=0
		\end{cases},
		\]
		which shows that $d\theta_1=-\theta_1\wedge\theta_2$ and $d\theta_2=0$. Hence, the equality \eqref{eq:comp} simplifies to
		\[
		\begin{cases}
			-\frac{1}{2}c_1\theta_1\wedge\theta_2=a_{11}\theta_1\wedge\theta_2-a_{22}\theta_1\wedge\theta_2-a_{11}\theta_1\wedge\theta_2=-a_{22}\theta_1\wedge\theta_2\\
			-\frac{1}{2}c_2\theta_1\wedge\theta_2=a_{12}\theta_1\wedge\theta_2
		\end{cases},
		\]
		hence $c_1=2a_{22}$ and $c_2=-2a_{12}$. It follows that $(\sigma,\psi)$ is a coboundary exactly when there exist $h_1,h_2\in C^{\infty}(M)$ and $a_{i,j}\in\mathbb{R}$ for $1\leq i,j\leq 2$ such that
		\[
		\begin{cases}
			\sigma_1=dh_1+h_2\omega_1-h_1\omega_2-a_{11}\omega_1-a_{21}\omega_2\\
			\sigma_2=dh_2-a_{12}\omega_1-a_{22}\omega_2\\
			c_1=2a_{22}\\
			c_2=-2a_{12}
		\end{cases},
		\]
		that is
		\[
		\begin{cases}
			\sigma_1=dh_1+h_2\omega_1-h_1\omega_2-a_{11}\omega_1-a_{21}\omega_2\\
			\sigma_2=dh_2+\frac{1}{2}c_2\omega_1-\frac{1}{2}c_1\omega_2
		\end{cases}.
		\]
		This in turn is equivalent with
		\[
		\begin{cases}
			\sigma_2-\frac{1}{2}c_2\omega_1+\frac{1}{2}c_1\omega_2=d\big(h_2-a_{11}\big)\\
			\sigma_1=\big(h_2-a_{11}\big)\omega_1+d\big(h_1+a_{21}\big)-\big(h_1+a_{21}\big)\omega_2
		\end{cases}.
		\]
		Setting $f:=h_2-a_{11}$ and $g:=-(h_1+a_{21})$, we obtain the expression \eqref{eq:coboundaries}. 
	\end{proof}
	
	We can express Prop.~\ref{prop:coho} in terms of the mapping cone of a suitable cochain map. Let us recall the definition and the sign conventions that we will use for mapping cones.
	
	\begin{defi}
		Let $(A^{\bullet},d_{A})$ and $(B^{\bullet},d_B)$ be complexes and $\Phi:(A^{\bullet},d_A)\rightarrow(B^{\bullet},d_B)$ a cochain map. The mapping cone $\mathcal{C}(\Phi)$ is the complex $\big(A^{\bullet}\oplus B^{\bullet-1},d_{\mathcal{C}}\big)$ with differential
		\[
		d_{\mathcal{C}}(a,b)=\big(d_{A}a,\Phi(a)-d_{B}b\big),\hspace{1cm}\text{for}\ a\in A^{k}, b\in B^{k-1}.
		\]
	\end{defi}
	
	In our setting, we have a natural cochain map which we now introduce. Recall that the Lie $\mathfrak{aff}(1)$-foliation $(\mathcal{F},\omega)$ is defined by $\omega_1,\omega_2\in\Omega^{1}(M)$ satisfying 
	\begin{equation*}\label{eq:lie-aff}
		\begin{cases}
			d\omega_1-\omega_2\wedge\omega_1=0\\
			d\omega_2=0
		\end{cases}.
	\end{equation*}
	This means that $\omega_1$ is a one-cocycle in the Morse-Novikov complex $\big(\Omega^{\bullet}(M),d_{\omega_2}\big)$.
	
	\begin{lemma}
		We have a cochain map
		\begin{equation*}\label{eq:cochain}
			\Phi:\big(\Omega^{\bullet}(M),d\big)\rightarrow\big(\Omega^{\bullet+1}(M),d_{\omega_2}\big):\alpha\mapsto\alpha\wedge\omega_1.
		\end{equation*}
	\end{lemma}
	\begin{proof}
		We compute for $\alpha\in\Omega^{k}(M)$:
		\begin{align*}
			d_{\omega_2}(\alpha\wedge\omega_1)&=d(\alpha\wedge\omega_1)-\omega_2\wedge\alpha\wedge\omega_1\\
			&=d\alpha\wedge\omega_1+(-1)^{k}\alpha\wedge d\omega_1-\omega_2\wedge\alpha\wedge\omega_1\\
			&=d\alpha\wedge\omega_1+(-1)^{k}\alpha\wedge\omega_2\wedge\omega_1-(-1)^{k}\alpha\wedge\omega_2\wedge\omega_1\\
			&=d\alpha\wedge\omega_1.\qedhere
		\end{align*}
	\end{proof}
	
	Its mapping cone $\mathcal{C}(\Phi)$ is the complex $\big(\Omega^{\bullet}(M)\oplus\Omega^{\bullet}(M),d_{\mathcal{C}}\big)$ with differential
	\[
	d_{\mathcal{C}}(\alpha,\beta)=\big(d\alpha,\alpha\wedge\omega_1-d\beta+\omega_2\wedge\beta\big),\hspace{1cm}\text{for}\ \alpha,\beta\in\Omega^{k}(M).
	\]
	Rephrasing Prop.~\ref{prop:coho}, one-cocycles in $\big(\mathcal{A}^{\bullet},D\big)$ are  $(\sigma_1,\sigma_2,c_1,c_2)\in\Omega^{1}(M)^{\times 2}\times\mathbb{R}^{2}$ such that
	\[
	d_{\mathcal{C}}\left(\sigma_2-\frac{1}{2}c_2\omega_1+\frac{1}{2}c_1\omega_2,\sigma_1\right)=0.
	\]
	A one-cocycle $(\sigma_1,\sigma_2,c_1,c_2)$ is exact in $\big(\mathcal{A}^{\bullet},D\big)$ iff. there exist $f,g\in C^{\infty}(M)$ such that
	\[
	\left(\sigma_2-\frac{1}{2}c_2\omega_1+\frac{1}{2}c_1\omega_2,\sigma_1\right)=d_{\mathcal{C}}(f,g).
	\]
	Consequently, we obtain the following description of the deformation cohomology $H^{1}(\mathcal{A})$.

	\begin{cor}\label{cor:isom}
		We have an isomorphism
		\[
		H^{1}(\mathcal{A})\overset{\sim}{\longrightarrow}H^{1}(\mathcal{C}(\Phi)):[(\sigma_1,\sigma_2,c_1,c_2)]\mapsto\left[\left(\sigma_2-\frac{1}{2}c_2\omega_1+\frac{1}{2}c_1\omega_2,\sigma_1\right)\right].
		\]
	\end{cor}
	%\begin{proof}
	%Prop.~\ref{prop:coho}
	%\end{proof}
	%\begin{cor}
	%One-cocycles in $\big(\mathcal{A}^{\bullet},D\big)$ are  $(\sigma_1,\sigma_2,c_1,c_2)\in\Omega^{1}(M)^{\times 2}\times\mathbb{R}^{2}$ such that
	%\[
	%d_{\mathcal{C}}\left(\sigma_2-\frac{1}{2}c_2\omega_1+\frac{1}{2}c_1\omega_2,\sigma_1\right)=0.
	%\]
	%A one-cocycle $(\sigma_1,\sigma_2,c_1,c_2)$ is exact in $\big(\mathcal{A}^{\bullet},D\big)$ iff. there exist $f,g\in C^{\infty}(M)$ such that
	%\[
	%\left(\sigma_2-\frac{1}{2}c_2\omega_1+\frac{1}{2}c_1\omega_2,\sigma_1\right)=d_{\mathcal{C}}(f,g).
	%\]
	%\end{cor}
	
	Hence, the rigidity statement in Thm.~\ref{thm:rigidity} reduces to the following when the Lie foliation $(\mathcal{F},\omega)$ is modeled on the Lie algebra $\mathfrak{aff}(1)$. 
	
	\begin{cor}\label{cor:rigid}
		Let $M$ be a compact manifold and $(\mathcal{F},\omega)$ a Lie $\mathfrak{aff}(1)$-foliation on $M$. If the mapping cone $\mathcal{C}(\Phi)$  satisfies $H^{1}(\mathcal{C}(\Phi))=0$, then $(\mathcal{F},\omega)$ is rigid.
	\end{cor}
	
	It remains to figure out when the cohomology group $H^{1}(\mathcal{C}(\Phi))$ vanishes. To this end, we note that the mapping cone fits in a short exact sequence of complexes
	\[
	\begin{tikzcd}
		0\arrow[r]&\big(\Omega^{\bullet}(M),-d_{\omega_2}\big)\arrow{r}{\gamma\mapsto(0,\gamma)}&\big(\mathcal{C}(\Phi)^{\bullet},d_{\mathcal{C}}\big)\arrow{r}{(\alpha,\beta)\mapsto\alpha}&\big(\Omega^{\bullet}(M),d\big)\arrow[r]&0.
	\end{tikzcd}
	\]
	We get a long exact sequence in cohomology
	\[
	\cdots\rightarrow H^{k}_{\omega_2}(M)\rightarrow H^{k}(\mathcal{C}(\Phi))\rightarrow H^{k}(M)\overset{[\Phi]}{\longrightarrow} H^{k+1}_{\omega_2}(M)\rightarrow H^{k+1}(\mathcal{C}(\Phi))\rightarrow H^{k+1}(M)\rightarrow\cdots
	\]
	with connecting homomorphism
	\[
	[\Phi]:H^{k}(M)\rightarrow H^{k+1}_{\omega_2}(M):[\alpha]\mapsto[\alpha\wedge\omega_1].
	\]
	We now investigate the first few terms of this long exact sequence, under the assumption that $M$ is compact and connected.
	\begin{enumerate}[i)]
		\item $H^{0}_{\omega_2}(M)=0$. 
		
		\noindent
		Indeed, since $M$ is compact and $\omega_2$ is nowhere zero, the class $[\omega_2]\in H^{1}(M)$ is non-trivial. Because $M$ is connected, this implies that $H^{0}_{\omega_2}(M)$ vanishes, see Prop.~\ref{prop:prelim}~$ii)$. 
		\item $H^{0}(\mathcal{C}(\Phi))=0$.
		
		\noindent
		To see this, note that
		\begin{align*}
			H^{0}(\mathcal{C}(\Phi))&=\left\{(f,g)\in C^{\infty}(M)\oplus C^{\infty}(M): df=0\ \text{and}\ f\omega_1=dg-g\omega_2\right\}\\
			&=\left\{(c,g)\in\mathbb{R}\oplus C^{\infty}(M): c\omega_1=dg-g\omega_2\right\}.
		\end{align*}
		However, if $c\omega_1=dg-g\omega_2$ then necessarily $c=0$. Indeed, by compactness of $M$ there exists $x\in M$ such that $(dg)_{x}=0$, which implies that $c(\omega_1)_x+g(x)(\omega_2)_x=0$. Because $\omega_1$ and $\omega_2$ are linearly independent at every point, this implies that $c=0$. But then $dg-g\omega_2=0$, which means that $g\in H^{0}_{\omega_2}(M)$. By item $i)$ above, also $g$ vanishes.
	\end{enumerate}
	
	It follows that the long exact sequence looks like
	\begin{equation}\label{eq:long-exact-sequence}
		0\rightarrow \mathbb{R}\overset{[\Phi]}{\longrightarrow} H^{1}_{\omega_2}(M)\rightarrow H^{1}(\mathcal{C}(\Phi))\rightarrow H^{1}(M)\overset{[\Phi]}{\longrightarrow}H^{2}_{\omega_2}(M)\rightarrow \cdots
	\end{equation}
	
	\begin{cor}\label{cor:injective}
		Let $M$ be compact and connected. The cohomology  $H^{1}(\mathcal{C}(\Phi))$ vanishes exactly when $H^{1}_{\omega_2}(M)=\mathbb{R}[\omega_1]$ and $[\Phi]:H^{1}(M)\rightarrow H^{2}_{\omega_2}(M)$ is injective.
	\end{cor}
	\begin{proof}
		First assume that $H^{1}(\mathcal{C}(\Phi))$ vanishes. The long exact sequence \eqref{eq:long-exact-sequence} then becomes 
		\begin{equation*}
			0\rightarrow \mathbb{R}\overset{[\Phi]}{\longrightarrow} H^{1}_{\omega_2}(M)\rightarrow 0\rightarrow H^{1}(M)\overset{[\Phi]}{\longrightarrow}H^{2}_{\omega_2}(M)\rightarrow \cdots
		\end{equation*}
		So $[\Phi]:\mathbb{R}\rightarrow H^{1}_{\omega_2}(M)$ is an isomorphism, i.e. $H^{1}_{\omega_2}(M)=\mathbb{R}[\omega_1]$, and $[\Phi]:H^{1}(M)\rightarrow H^{2}_{\omega_2}(M)$ is injective. Conversely, assume that $H^{1}_{\omega_2}(M)=\mathbb{R}[\omega_1]$ and $[\Phi]:H^{1}(M)\rightarrow H^{2}_{\omega_2}(M)$ is injective. It follows that the map $H^{1}_{\omega_2}(M)\rightarrow H^{1}(\mathcal{C}(\Phi))$ in \eqref{eq:long-exact-sequence} is identically zero, hence the map $H^{1}(\mathcal{C}(\Phi))\rightarrow H^{1}(M)$ is injective. On the other hand, the map $H^{1}(\mathcal{C}(\Phi))\rightarrow H^{1}(M)$ is identically zero because $[\Phi]:H^{1}(M)\rightarrow H^{2}_{\omega_2}(M)$ is injective. Hence, $H^{1}(\mathcal{C}(\Phi))$ vanishes.
	\end{proof}
	
	We can now rephrase Cor.~\ref{cor:rigid} to obtain the main result of this subsection.
	
	\begin{cor}\label{cor:req-rigidity}
		Let $M$ be compact and connected, with a Lie $\mathfrak{aff}(1)$-foliation $(\mathcal{F},\omega)$ defined by one-forms $\omega_1,\omega_2\in\Omega^{1}(M)$ satisfying
		\begin{equation}\label{eq:aff-fol}
			\begin{cases}
				d\omega_1-\omega_2\wedge\omega_1=0\\
				d\omega_2=0
			\end{cases}.
		\end{equation}
		If $H^{1}_{\omega_2}(M)=\mathbb{R}[\omega_1]$ and $\bullet\wedge[\omega_1]:H^{1}(M)\rightarrow H^{2}_{\omega_2}(M)$ is injective, then $(\mathcal{F},\omega)$ is rigid.
	\end{cor}

	\subsection{Model Lie affine foliations}\label{subsec:model}
	In this subsection, we recall an important construction that yields Lie affine foliations on mapping tori. We will call them \emph{model Lie affine foliations}. These are particularly important because any Lie affine foliation on a compact, connected, orientable manifold of dimension $3$ or $4$ is foliated diffeomorphic to a foliation of this type.
	
	In \S\ref{subsub:rigidity-models}, we find explicit criteria for model Lie affine foliations under which the conditions in Cor.~\ref{cor:req-rigidity} are satisfied. This yields many examples of Lie affine foliations that are rigid, when deformed within the class of Lie foliations. By contrast, we show in \S\ref{subsub:non-rigid} that model Lie affine foliations are never rigid when deformed merely as foliations. In \S\ref{subsub:independence} we prove a remarkable auxiliary result: for model Lie affine foliations, the cohomological conditions in Cor.~\ref{cor:req-rigidity} turn out to be independent of the choice of defining Maurer-Cartan form, i.e. they only depend on the underlying foliation. We will use this to deduce a general rigidity result for Lie affine foliations on compact, connected, orientable manifolds of dimension $3$ or $4$ which, as mentioned above, are foliated diffeomorphic to model Lie affine foliations.

	\subsubsection{Model Lie affine foliations}
	Let $L$ be a compact, connected manifold and assume that $\varphi\in\text{Diff}(L)$ is a diffeomorphism. Consider the mapping torus
	\[
	\mathbb{T}_{\varphi}:=\frac{L\times \mathbb{R}}{(x,t)\sim(\varphi(x),t+1)}.
	\]
	Assume that $\alpha\in\Omega^{1}(L)$ is a closed nowhere-vanishing one-form satisfying $\varphi^{*}\alpha=e^{\lambda}\alpha$ for some $\lambda\neq 0$. It follows that we get well-defined one-forms $\omega_1,\omega_2\in\Omega^{1}(\mathbb{T}_{\varphi})$ given by
	\begin{equation}\label{eq:model-forms}
		\begin{cases}
			\omega_1:=e^{-\lambda t}\alpha\\
			\omega_2=-\lambda dt
		\end{cases}.
	\end{equation}
	These satisfy the equations \eqref{eq:aff-fol}, so they define a Lie $\mathfrak{aff}(1)$-foliation $(\mathcal{G}_{\alpha,\lambda},\omega)$. 
	
	\begin{defi}\label{def:model}
		A model Lie affine foliation is a Lie $\mathfrak{aff}(1)$-foliation of the form $(\mathcal{G}_{\alpha,\lambda},\omega)$.
	\end{defi}
	
	Any compact, connected, orientable manifold of dimension $3$ or $4$ with a Lie $\mathfrak{aff}(1)$-foliation is foliated diffeomorphic to a foliated manifold of the type $(\mathbb{T}_{\varphi},\mathcal{G}_{\alpha,\lambda})$ by \cite{caron},\cite{Matsumoto}. 
	
	\subsubsection{Rigidity of model Lie affine foliations}\label{subsub:rigidity-models}
	Let $(\mathcal{G}_{\alpha,\lambda},\omega)$ be a model Lie affine foliation on the mapping torus $\mathbb{T}_{\varphi}$. Below, we give sufficient conditions ensuring that the requirements of Cor.~\ref{cor:req-rigidity} are met. To do so, we first recall that $H^{1}(\mathbb{T}_{\varphi})$ fits in a short exact sequence
\begin{equation}\label{eq:first-sequence}
\begin{tikzcd}
	0\arrow[r] & \mathbb{R}\arrow{r}{c\mapsto c[dt]} & H^{1}(\mathbb{T}_{\varphi})\arrow{r}{[i_{0}^{*}]} & \ker\left([\varphi^{*}-\mathrm{Id}]:H^{1}(L)\rightarrow H^{1}(L)\right)\arrow[r] &0
\end{tikzcd},
\end{equation}
where $i_{0}:L\rightarrow\mathbb{T}_{\varphi}:x\mapsto [(x,0)]$ is the inclusion of $L$ as the zero-fiber. See \cite[Lemma~12]{bazzoni}.
Because of Cor.~\ref{cor:computed}, the Morse-Novikov cohomology $H^{2}_{-\lambda dt}(\mathbb{T}_{\varphi})$ fits in a similar sequence. To describe it, we introduce for every $c\in\mathbb{R}$ the spaces
\begin{align*}
&K^{\bullet}_c:=\ker\left([\varphi^{*}-e^{-c}\mathrm{Id}]:H^{\bullet}(L)\rightarrow H^{\bullet}(L)\right),\\
&C^{\bullet}_c:=\mathrm{coker}\left([\varphi^{*}-e^{-c}\mathrm{Id}]:H^{\bullet}(L)\rightarrow H^{\bullet}(L)\right).
\end{align*}
We then have another short exact sequence
\begin{equation}\label{eq:second-sequence}
\begin{tikzcd}
	0\arrow[r] & C_{-\lambda}^{1}\arrow{r} & H^{2}_{-\lambda dt}(\mathbb{T}_{\varphi})\arrow{r}{[i_{0}^{*}]} & K_{-\lambda}^{2}\arrow[r] &0.
\end{tikzcd}
\end{equation}
Note that the surjection $H^{2}_{-\lambda dt}(\mathbb{T}_{\varphi})\rightarrow K_{-\lambda}^{2}$ in this sequence is indeed given by $[i_{0}^{*}]$. This is a consequence of the short exact sequence \eqref{eq:sh-ex} and Prop.~\ref{prop:ker}. The following result provides a relation between the short exact sequence \eqref{eq:first-sequence} and the short exact sequence \eqref{eq:second-sequence}.
	
\begin{lemma}
Consider a model Lie affine foliation $(\mathcal{G}_{\alpha,\lambda},\omega)$ on the mapping torus $\mathbb{T}_{\varphi}$. We then have a commutative diagram with exact rows, given by
\begin{equation}\label{eq:diagram-sequences}
\begin{tikzcd}	
	0\arrow[r] & \mathbb{R}\arrow{r}\arrow{d}{c\mapsto \frac{c}{\lambda}\overline{[\alpha]}} & H^{1}(\mathbb{T}_{\varphi})\arrow{r}{[i_{0}^{*}]}\arrow{d}{\bullet\wedge[e^{-\lambda t}\alpha]} & K^{1}_{0}\arrow[r]\arrow{d}{\bullet\wedge[\alpha]} &0\\
	0\arrow[r] & C_{-\lambda}^{1}\arrow{r} & H^{2}_{-\lambda dt}(\mathbb{T}_{\varphi})\arrow{r}{[i_{0}^{*}]} & K_{-\lambda}^{2}\arrow[r] &0
\end{tikzcd}.
\end{equation}
\end{lemma}
\begin{proof}
Clearly, the second square commutes because of the equality $i_{0}^{*}(e^{-\lambda t}\alpha)=\alpha$. To show that the first square commutes, we argue as follows. For $c\in\mathbb{R}$, the class $c[dt]\in H^{1}(\mathbb{T}_{\varphi})$ gets mapped by $\bullet\wedge[e^{-\lambda t}\alpha]$ to $[cdt\wedge e^{-\lambda t}\alpha]\in H^{2}_{-\lambda dt}(\mathbb{T}_{\varphi})$. The latter lies in the kernel of $[i_{0}^{*}]$, hence it comes from a uniquely determined element in $C_{-\lambda}^{1}=\mathrm{coker}\left([\varphi^{*}-e^{\lambda}\mathrm{Id}]:H^{1}(L)\rightarrow H^{1}(L)\right)$.
To describe this element, let us denote by $\mathcal{V}$ the foliation by fibers of $\mathbb{T}_{\varphi}\rightarrow S^{1}$. According to the short exact sequence \eqref{eq:sh-ex}, the class $[cdt\wedge e^{-\lambda t}\alpha]\in H^{2}_{-\lambda dt}(\mathbb{T}_{\varphi})$ comes from an element in $\text{coker}\big(\partial: H^{1}(\mathcal{V})\rightarrow H^{1}(\mathcal{V})\big)$ via the map
\[
j:\Omega^{1}(\mathcal{V})\rightarrow\Omega^{2}(\mathbb{T}_{\varphi}):\beta\mapsto \lambda dt\wedge\widetilde{\beta},
\]
where $\widetilde{\beta}\in\Omega^{1}(\mathbb{T}_{\varphi})$ is any extension of $\beta\in\Omega^{1}(\mathcal{V})$. Clearly this element must be
\[
\left[\frac{c}{\lambda}e^{-\lambda t}\alpha\right]\ \text{mod}\ \text{im}(\partial)\in\text{coker}\big(\partial: H^{1}(\mathcal{V})\rightarrow H^{1}(\mathcal{V})\big).
\]
By Prop.~\ref{prop:coker} $ii)$, it corresponds to a class in $\mathrm{coker}\left([\varphi^{*}-e^{\lambda}\mathrm{Id}]:H^{1}(L)\rightarrow H^{1}(L)\right)$, namely the equivalence class of the element
\[
\int_{0}^{1}e^{\lambda s}\left[\frac{c}{\lambda}e^{-\lambda s}\alpha\right]ds=\frac{c}{\lambda}[\alpha]\in H^1(L).
\]
We now showed that the first vertical arrow in the diagram \eqref{eq:diagram-sequences} is defined appropriately, so that the first square in the diagram \eqref{eq:diagram-sequences} commutes. This proves the statement.
\end{proof}

As a consequence, we obtain a sufficient condition regarding the second requirement in Cor.~\ref{cor:req-rigidity}, namely injectivity of the map $\bullet\wedge[\omega_1]:H^{1}(\mathbb{T}_{\varphi})\rightarrow H^{2}_{\omega_{2}}(\mathbb{T}_{\varphi})$.

\begin{cor}\label{cor:cond-injectivity}
Consider a model Lie affine foliation $(\mathcal{G}_{\alpha,\lambda},\omega)$ on the mapping torus $\mathbb{T}_{\varphi}$. Assume that the following conditions are satisfied:
\begin{enumerate}
	\item $[\alpha]\notin\text{Im}\left([\varphi^{*}-e^{\lambda}\mathrm{Id}]:H^{1}(L)\rightarrow H^{1}(L)\right),$
	\item $\bullet\wedge[\alpha]:\ker\left([\varphi^{*}-\mathrm{Id}]:H^{1}(L)\rightarrow H^{1}(L)\right)\rightarrow H^{2}(L)$ is injective.
\end{enumerate}
It then follows that the map $\bullet\wedge[\omega_1]:H^{1}(\mathbb{T}_{\varphi})\rightarrow H^{2}_{\omega_{2}}(\mathbb{T}_{\varphi})$ is also injective.
\end{cor}
\begin{proof}
The assumptions state that the leftmost and rightmost vertical maps in the diagram \eqref{eq:diagram-sequences} are injective. The five lemma \cite[Chapter III, Ex.~15]{Lang} implies that also the vertical map in the middle $\bullet\wedge[e^{-\lambda t}\alpha]:H^{1}(\mathbb{T}_{\varphi})\rightarrow H^{2}_{-\lambda dt}(\mathbb{T}_{\varphi})$ is injective. This proves the statement.
\end{proof}
	
We obtain the following criterion for rigidity of model Lie affine foliations, which is useful in practice. It gives sufficient conditions ensuring that the requirements of Cor.~\ref{cor:req-rigidity} hold.
	
	\begin{cor}\label{cor:rigidity-model-interim}
	Consider a model Lie affine foliation $(\mathcal{G}_{\alpha,\lambda},\omega)$ on the mapping torus $\mathbb{T}_{\varphi}$. Assume that the following conditions are satisfied:
	\begin{enumerate}
		\item $\ker\left([\varphi^{*}-e^{\lambda}\mathrm{Id}]^{2}:H^{1}(L)\rightarrow H^{1}(L)\right)=\mathbb{R}[\alpha]$,
		\item $\bullet\wedge[\alpha]:\ker\left([\varphi^{*}-\mathrm{Id}]:H^{1}(L)\rightarrow H^{1}(L)\right)\rightarrow H^{2}(L)$ is injective.
	\end{enumerate}
	Then $H^{1}_{\omega_2}(\mathbb{T}_{\varphi})=\mathbb{R}[\omega_1]$ and $\bullet\wedge[\omega_1]:H^{1}(\mathbb{T}_{\varphi})\rightarrow H^{2}_{\omega_2}(\mathbb{T}_{\varphi})$ is injective, so $(\mathcal{G}_{\alpha,\lambda},\omega)$ is rigid.
	\end{cor}
	\begin{proof}
	The assumption that $\ker\left([\varphi^{*}-e^{\lambda}\mathrm{Id}]^{2}:H^{1}(L)\rightarrow H^{1}(L)\right)=\mathbb{R}[\alpha]$ implies that
	\begin{equation}\label{eq:conditions}
	\begin{cases}
		\ker\left([\varphi^{*}-e^{\lambda}\mathrm{Id}]:H^{1}(L)\rightarrow H^{1}(L)\right)=\mathbb{R}[\alpha]\\
		[\alpha]\notin\text{Im}\left([\varphi^{*}-e^{\lambda}\mathrm{Id}]:H^{1}(L)\rightarrow H^{1}(L)\right)
	\end{cases}.
	\end{equation}
	Now note that Cor.~\ref{cor:computed} $ii)$ gives an isomorphism in degree one
	\[
	H^{1}_{-\lambda dt}(\mathbb{T}_{\varphi})\overset{\sim}{\longrightarrow}\ker\left([\varphi^{*}-e^{\lambda}\mathrm{Id}]:H^{1}(L)\rightarrow H^{1}(L)\right),
	\]
	which maps $[e^{-\lambda t}\alpha]$ to $[\alpha]$. By the first condition in \eqref{eq:conditions}, we then get $H^{1}_{-\lambda dt}(\mathbb{T}_{\varphi})=\mathbb{R}[e^{-\lambda t}\alpha]$. Next, the assumption that $\bullet\wedge[\alpha]:\ker\left([\varphi^{*}-\mathrm{Id}]:H^{1}(L)\rightarrow H^{1}(L)\right)\rightarrow H^{2}(L)$ is injective along with the second condition in \eqref{eq:conditions} imply that $\bullet\wedge[e^{-\lambda t}\alpha]:H^{1}(\mathbb{T}_{\varphi})\rightarrow H^{2}_{-\lambda dt}(\mathbb{T}_{\varphi})$ is injective, by Cor.~\ref{cor:cond-injectivity}. So the conditions of Cor.~\ref{cor:req-rigidity} are met, hence $(\mathcal{G}_{\alpha,\lambda},\omega)$ is rigid.
	\end{proof}
	
	Rephrasing the first condition in Cor.~\ref{cor:rigidity-model-interim} yields a criterion that is even easier to check.

	\begin{cor}\label{cor:rigidity-model}
	Consider a model Lie affine foliation $(\mathcal{G}_{\alpha,\lambda},\omega)$ on the mapping torus $\mathbb{T}_{\varphi}$. Assume that the following conditions are satisfied:
	\begin{enumerate}
		\item The eigenvalue $e^{\lambda}$ of $[\varphi^{*}]:H^{1}(L)\rightarrow H^{1}(L)$ has algebraic multiplicity equal to $1$,
		\item The map $\bullet\wedge[\alpha]:\ker\left([\varphi^{*}-\mathrm{Id}]:H^{1}(L)\rightarrow H^{1}(L)\right)\rightarrow H^{2}(L)$ is injective.
	\end{enumerate}
	Then $H^{1}_{\omega_2}(\mathbb{T}_{\varphi})=\mathbb{R}[\omega_1]$ and $\bullet\wedge[\omega_1]:H^{1}(\mathbb{T}_{\varphi})\rightarrow H^{2}_{\omega_2}(\mathbb{T}_{\varphi})$ is injective, so $(\mathcal{G}_{\alpha,\lambda},\omega)$ is rigid.
	\end{cor}
	\begin{proof}
	For a model Lie affine foliation $(\mathcal{G}_{\alpha,\lambda},\omega)$ on $\mathbb{T}_{\varphi}$, the following are equivalent:
	\begin{enumerate}[i)]
		\item The eigenvalue $e^{\lambda}$ of $[\varphi^{*}]:H^{1}(L)\rightarrow H^{1}(L)$ has algebraic multiplicity equal to $1$,
		\item $\ker\left([\varphi^{*}-e^{\lambda}\mathrm{Id}]^{2}:H^{1}(L)\rightarrow H^{1}(L)\right)=\mathbb{R}[\alpha]$.
	\end{enumerate}
    To see this, denote by $b_{1}(L)$ the first Betti number of $L$. We will show in Lemma~\ref{lem:dim} that $b_{1}(L)\geq 2$. We get a chain of generalized eigenspaces in the vector space $H^{1}(L)$, given by
    \begin{equation}\label{eq:chain}
    \ker\big([\varphi^{*}-e^{\lambda}\mathrm{Id}]\big)\subset \ker\big([\varphi^{*}-e^{\lambda}\mathrm{Id}]^{2}\big)\subset\cdots\subset \ker\big([\varphi^{*}-e^{\lambda}\mathrm{Id}]^{b_{1}(L)}\big).
    \end{equation}
    It is well-known that the dimension of $\ker\big([\varphi^{*}-e^{\lambda}\mathrm{Id}]^{b_{1}(L)}\big)$ coincides with the algebraic multiplicity of the eigenvalue $e^{\lambda}$. We can now show that items $i)$ and $ii)$ above are equivalent.
    
    If the algebraic multiplicity of $e^{\lambda}$ is equal to $1$, then necessarily the chain \eqref{eq:chain} is constant equal to $\mathbb{R}[\alpha]$. In particular, we get that $\ker\big([\varphi^{*}-e^{\lambda}\mathrm{Id}]^{2}\big)=\mathbb{R}[\alpha]$. Conversely, assume that $\ker\big([\varphi^{*}-e^{\lambda}\mathrm{Id}]^{2}\big)=\mathbb{R}[\alpha]$. We then have that $\ker\big([\varphi^{*}-e^{\lambda}\mathrm{Id}]\big)=\ker\big([\varphi^{*}-e^{\lambda}\mathrm{Id}]^{2}\big)=\mathbb{R}[\alpha]$. So the chain \eqref{eq:chain} is constant equal to $\mathbb{R}[\alpha]$ and therefore $e^{\lambda}$ has algebraic multiplicity $1$.
	\end{proof}

	\subsubsection{Independence of Maurer-Cartan form}\label{subsub:independence}
	The following results are motivated by the study of rigidity in the particular setting of Lie affine foliations $(\mathcal{F},\zeta)$ on connected, compact, orientable manifolds $M$ of dimension $3$ or $4$. Say that $\mathcal{F}$ is defined by forms $\zeta_1,\zeta_2\in\Omega^{1}(M)$. The normal form results \cite{caron},\cite{Matsumoto} yield a foliated diffeomorphism $\phi:(\mathbb{T}_{\varphi},\mathcal{G}_{\alpha,\lambda})\rightarrow (M,\mathcal{F})$ for a certain model Lie affine foliation $(\mathbb{T}_{\varphi},\mathcal{G}_{\alpha,\lambda},\omega)$. This does not imply that $\phi^{*}\zeta_1$ and $\phi^{*}\zeta_2$ agree with the forms $\omega_1$ and 
	$\omega_2$ defined in \eqref{eq:model-forms}: these two pairs of forms merely define the same foliation $\mathcal{G}_{\alpha,\lambda}$. We now show that  $\phi^{*}\zeta_1$ and $\phi^{*}\zeta_2$ cannot differ too much from $\omega_1$ and $\omega_2$ defined in \eqref{eq:model-forms}. As a consequence, we can show that the pair $(\zeta_1,\zeta_2)$ satisfies the requirements of Cor.~\ref{cor:req-rigidity} if and only if $(\omega_1,\omega_2)$ does. This will allow us to deduce general rigidity results in dimension $3$ and $4$ from rigidity of the models  $(\mathcal{G}_{\alpha,\lambda},\omega)$ in question.

	\begin{prop}\label{prop:other-forms}
		Consider a model Lie affine foliation $(\mathcal{G}_{\alpha,\lambda},\omega)$. Let $\zeta_1,\zeta_2\in\Omega^{1}(\mathbb{T}_{\varphi})$ be independent one-forms such that $T\mathcal{G}_{\alpha,\lambda}=\ker\zeta_1\cap\ker\zeta_2$ and
		\begin{equation*}
			\begin{cases}
				d\zeta_1-\zeta_2\wedge\zeta_1=0\\
				d\zeta_2=0
			\end{cases}.
		\end{equation*}
		Then there exist $f,g\in C^{\infty}(S^{1})$ and $K\in\mathbb{R}_{0}$ such that
		\begin{equation*}
			\begin{cases}
				\zeta_2=\omega_2+df\\
				\zeta_1=e^{f}\left(K\omega_1+dg-g\omega_2\right)
			\end{cases}.
		\end{equation*}
	\end{prop}
	The statement says that $\zeta_2$ is basic with respect to the fiber foliation of $\mathbb{T}_{\varphi}\rightarrow S^{1}$ and that $[\zeta_2]=[\omega_2]$ in $H^{1}(S^{1})$. Moreover, a choice of primitive $f\in C^{\infty}(S^1)$ for $\zeta_2-\omega_2$ yields an isomorphism (see Prop.~\ref{prop:prelim}~$i)$)
	\begin{equation*}
		H^{\bullet}_{\zeta_2}(\mathbb{T}_\varphi)\rightarrow H^{\bullet}_{\omega_2}(\mathbb{T}_\varphi):[\eta]\mapsto [e^{-f}\eta],
	\end{equation*}
	under which $[\zeta_1]\in H^{1}_{\zeta_2}(\mathbb{T}_\varphi)$ corresponds with a non-zero multiple of $[\omega_1]\in H^{1}_{\omega_2}(\mathbb{T}_\varphi)$.
	
	\begin{proof}[Proof of Prop.~\ref{prop:other-forms}]
		We break down the proof into a few steps.
		
		\vspace{0.2cm}
		\noindent
		\underline{Step 1:} For every fiber $L$ of $\mathbb{T}_{\varphi}\rightarrow S^{1}$, the foliation $\mathcal{G}_{\alpha,\lambda}|_{L}$ has dense leaves.
		
		\vspace{0.1cm}
		Note that the foliation $\mathcal{G}_{\alpha,\lambda}|_{L}$ is given by the kernel of the closed one-form $\alpha\in\Omega^{1}(L)$. The claim follows if we show that the period group of $\alpha$ cannot be discrete. If it were discrete, then the foliation integrating $\ker\alpha$ would be given by a fibration $p:L\rightarrow S^{1}$. Since $\varphi^{*}\alpha=e^{\lambda}\alpha$, we have that $\varphi\in\text{Diff}(L)$ preserves the fibers of $p:L\rightarrow S^{1}$. It follows that $\varphi$ descends to a diffeomorphism $\overline{\varphi}\in\text{Diff}(S^{1})$ and $\alpha$ descends to $\overline{\alpha}\in\Omega^{1}(S^{1})$, such that $\overline{\varphi}^{*}\overline{\alpha}=e^{\lambda}\overline{\alpha}$. But then we would have
		\begin{equation}\label{eq:int}
			\int_{S^{1}}\overline{\alpha}=\int_{S^{1}}\overline{\varphi}^{*}\overline{\alpha}=e^{\lambda}\int_{S^{1}}\overline{\alpha}.
		\end{equation}
		Note that $\overline{\alpha}$ is not exact -- otherwise $\alpha\in\Omega^{1}(L)$ would also be exact, which is impossible since it is nowhere zero. Hence, the integral of $\overline{\alpha}$ in \eqref{eq:int} is nonzero, and therefore $\lambda=0$. This is a contradiction, hence the period group of $\alpha$ cannot be discrete.
		
		\vspace{0.2cm}
		\noindent
		\underline{Step 2:} $\zeta_2$ is basic with respect to the fiber foliation of $\mathbb{T}_{\varphi}\rightarrow S^{1}$.
		
		\vspace{0.1cm}
		The pullback of $\zeta_2$ to each fiber $L$ vanishes on the leaves of $\mathcal{G}_{\alpha,\lambda}|_{L}$, hence we can write $i_{L}^{*}\zeta_2=h\alpha$ for some $h\in C^{\infty}(L)$. Since both $i_{L}^{*}\zeta_2$ and $\alpha$ are closed, it follows that $dh\wedge\alpha=0$, which implies that $h$ is basic with respect to the foliation $\mathcal{G}_{\alpha,\lambda}|_{L}$ defined by $\alpha$. Since the leaves of $\mathcal{G}_{\alpha,\lambda}|_{L}$ are dense by Step 1, we get that $h$ is constant. Consequently, on the covering space $L\times\mathbb{R}$ of $\mathbb{T}_{\varphi}$, we can write
		\begin{equation}\label{eq:covering}
			\zeta_2=f(t)\alpha+g(x,t)dt
		\end{equation}
		for some $f\in C^{\infty}(\mathbb{R})$ and $g\in C^{\infty}(L\times\mathbb{R})$. We still have to impose that the form \eqref{eq:covering} is closed and that it descends to $\mathbb{T}_{\varphi}$, which amounts to invariance under the identification $(x,t)\sim(\varphi(x),t+1)$. The condition that $\zeta_2$ is closed yields
		\[
		0=\left(dg-f'(t)\alpha\right)\wedge dt.
		\]
		This means that the pullback $\iota_{L\times\{t\}}^{*}\left(dg-f'(t)\alpha\right)$ to any fiber $L\times\{t\}$ vanishes. Because $\iota_{L\times\{t\}}^{*}dg$ vanishes at some point by compactness of $L$ and $\alpha$ is nowhere vanishing, this implies that $f'(t)=0$. In turn, we also get that $g$ only depends on the $t$-coordinate. Hence,
		\[
		\zeta_2=K\alpha+g(t)dt
		\]
		for $K\in\mathbb{R}$ and $g\in C^{\infty}(\mathbb{R})$, which is an equality of forms on $L\times\mathbb{R}$. Requiring that $\zeta_2$ descends to $\mathbb{T}_{\varphi}$ gives $\varphi^{*}(K\alpha)=K\alpha$ and $g(t+1)=g(t)$ for all $t\in\mathbb{R}$. However, since $\varphi^{*}\alpha=e^{\lambda}\alpha$ with $\lambda\neq 0$, it follows that $K=0$. Therefore, $\zeta_2=g(t)dt$ for some $g\in C^{\infty}(S^1)$.
		
		\vspace{0.2cm}
		\noindent
		\underline{Step 3:} We have that $\zeta_2=\omega_2+df$ for some $f\in C^{\infty}(S^1)$.
		
		\vspace{0.1cm}
		As $\zeta_2$ is closed and basic with respect to $\mathbb{T}_{\varphi}\rightarrow S^{1}$ by Step 2, we can write $\zeta_2=cdt+df$ for some $c\in\mathbb{R}_{0}$ and $f\in C^{\infty}(S^1)$. We have to show that necessarily $c=-\lambda$. 
		
		On the one hand, recall that $\zeta_1$ defines a non-zero class in $H^{1}_{\zeta_2}(\mathbb{T}_\varphi)$ (see item ii) in the text preceding Cor.~\ref{cor:injective}). 
		Hence, under the isomorphism 
		\begin{equation*}
			H^{\bullet}_{cdt+df}(\mathbb{T}_\varphi)\rightarrow H^{\bullet}_{cdt}(\mathbb{T}_\varphi):[\eta]\mapsto [e^{-f}\eta],
		\end{equation*}
		we get a non-zero class $[e^{-f}\zeta_1]\in H^{1}_{cdt}(\mathbb{T}_\varphi)$. Combining Cor.~\ref{cor:nabla} and Prop.~\ref{prop:ker}, we have an isomorphism
		\begin{equation}
			H^{1}_{cdt}(\mathbb{T}_\varphi)\rightarrow\ker\left([\varphi^{*}-e^{-c}\mathrm{Id}]:H^{1}(L)\rightarrow H^{1}(L)\right):[\eta]\mapsto \sigma_{r(\eta)}(0),
		\end{equation}
		where $r:\Omega^{\bullet}(\mathbb{T}_{\varphi})\rightarrow\Omega^{\bullet}(\mathcal{V})$ is the restriction to the fiber foliation $\mathcal{V}$ of $\mathbb{T}_{\varphi}\rightarrow S^{1}$. Hence, we get that $\sigma_{r(e^{-f}\zeta_1)}(0)\in H^{1}(L)$ is a non-zero element in the $e^{-c}$-eigenspace of $[\varphi^{*}]$.
		
		On the other hand, we know that $d\zeta_1-\zeta_2\wedge\zeta_1=0$ and that the kernel of $\zeta_2$ is the vertical distribution $\mathcal{V}$ of $\mathbb{T}_{\varphi}\rightarrow S^{1}$. It follows that $r(\zeta_1)\in\Omega^{1}(\mathcal{V})$ is closed on every fiber $L$ of $\mathbb{T}_{\varphi}\rightarrow S^{1}$. Since $f\in C^{\infty}(S^{1})$, the same holds for $r(e^{-f}\zeta_1)\in\Omega^{1}(\mathcal{V})$. Since $\zeta_1$ vanishes on  $\mathcal{G}_{\alpha,\lambda}$ and the leaves of $\mathcal{G}_{\alpha,\lambda}|_{L}$ are dense in each fiber $L$, it follows that $r(e^{-f}\zeta_1)|_{L}=K\alpha$ for some $K\in\mathbb{R}$. In particular, the class $\sigma_{r(e^{-f}\zeta_1)}(0)\in H^{1}(L)$ is a multiple of $[\alpha]$, hence it belongs to the $e^{\lambda}$-eigenspace of $[\varphi^{*}]$. Altogether, this shows that necessarily $c=-\lambda$.
		
		\vspace{0.2cm}
		\noindent
		\underline{Step 4:} We have that $\zeta_1=e^{f}\left(K\omega_1+dg-g\omega_2\right)$ for some $K\in\mathbb{R}_{0}$ and $g\in C^{\infty}(S^1)$.
		
		\vspace{0.1cm}
		
		By the last paragraph in Step 3, we can write 
		\[
		e^{-f}\zeta_1=h(t)\alpha+l(x,t)dt,
		\]
		which is an equality of forms on $L\times\mathbb{R}$. We also know that this form is closed with respect to the Morse-Novikov differential $d_{-\lambda dt}$ and that it descends to $\mathbb{T}_{\varphi}$. Closedness with respect to $d_{-\lambda dt}$ yields
		\[
		\left(dl-h'(t)\alpha\right)\wedge dt+\lambda dt\wedge h(t)\alpha=0.
		\]
		This means that the pullback $\iota_{L\times\{t\}}^{*}\left(dl-h'(t)\alpha-\lambda h(t)\alpha\right)$ to any fiber $L\times\{t\}$ vanishes. Because $\iota_{L\times\{t\}}^{*}dl$ vanishes at some point by compactness of $L$ and $\alpha$ is nowhere vanishing, this implies that
		$$
		h'(t)=-\lambda h(t),
		$$
		and therefore $h(t)=K e^{-\lambda t}$ for some $K\in\mathbb{R}_{0}$. It also follows that the function $l$ only depends on the $t$-coordinate. Hence,
		\[
		e^{-f}\zeta_1=K e^{-\lambda t}\alpha+l(t)dt.
		\]
		Since this form descends to $\mathbb{T}_{\varphi}$, it is invariant under the identification $(x,t)\sim(\varphi(x),t+1)$. It follows that $l(t+1)=l(t)$ for all $t\in\mathbb{R}$, i.e. $l\in C^{\infty}(S^{1})$.  The proof is finished if we can find $g\in C^{\infty}(S^{1})$ such that
		\begin{equation}\label{eq:ode}
			l=g'+\lambda g.
		\end{equation}
		Indeed, we then have that
		\begin{align*}
			\zeta_1&=e^{f}\left(K e^{-\lambda t}\alpha+g'(t)dt+\lambda g(t)dt\right)\\
			&=e^{f}\left(K\omega_1+dg-g\omega_2\right).
		\end{align*}
		In Lemma~\ref{lem:ode2} below we confirm that the differential equation \eqref{eq:ode} can indeed be solved for $g\in C^{\infty}(S^1)$. This finishes the last step of the proof.
	\end{proof}

	\begin{lemma}\label{lem:ode2}
		For every $\lambda\in\mathbb{R}_{0}$ and $l\in C^{\infty}(S^1)$, there is a unique function $g\in C^{\infty}(S^1)$ such that 
		\begin{equation}\label{eq:to-solve}
			l=g'+\lambda g.
		\end{equation}
	\end{lemma}
	\begin{proof}
		To prove uniqueness, it suffices to show that the homogeneous equation
		\begin{equation}\label{eq:homogeneous}
			g'+\lambda g=0
		\end{equation}
		only admits the zero solution $g\equiv 0$. To do so, note that any $g\in C^{\infty}(S^{1})$ reaches a minimum $m$ and a maximum $M$. If moreover $g$ satisfies the equation \eqref{eq:homogeneous}, then we get
		\[
		m\lambda=M\lambda=0.
		\]
		Since $\lambda\neq 0$, it follows that $m=M=0$ and therefore $g\equiv 0$.
		
		To prove existence, we already know from Lemma~\ref{lem:ode} that every not necessarily periodic solution $g\in C^{\infty}(\mathbb{R})$ to \eqref{eq:to-solve} is of the form
		\[
		g(t)=e^{-\lambda t}\left(\int_{0}^{t}l(s)e^{\lambda s}ds+C\right)
		\]
		for some $C\in\mathbb{R}$. We now show that we can fix the constant $C$ so that $g$ becomes periodic. Note that $g(t+1)=g(t)$ for all $t\in\mathbb{R}$ exactly when
		\begin{align}\label{eq:integrals}
			&e^{-\lambda}\left(\int_{0}^{t+1}l(s)e^{\lambda s}ds+C\right)=\int_{0}^{t}l(s)e^{\lambda s}ds+C\nonumber\\
			\Leftrightarrow &\ e^{-\lambda}\left(\int_{0}^{1}l(s)e^{\lambda s}ds\right)+e^{-\lambda}\left(\int_{1}^{t+1}l(s)e^{\lambda s}ds\right)+e^{-\lambda}C=\int_{0}^{t}l(s)e^{\lambda s}ds+C\nonumber\\
			\Leftrightarrow &\ e^{-\lambda}\left(\int_{0}^{1}l(s)e^{\lambda s}ds\right)+e^{-\lambda}\left(\int_{0}^{t}l(z+1)e^{\lambda z + \lambda}dz\right)+e^{-\lambda}C=\int_{0}^{t}l(s)e^{\lambda s}ds+C\nonumber\\
			\Leftrightarrow &\ e^{-\lambda}\left(\int_{0}^{1}l(s)e^{\lambda s}ds\right)+\int_{0}^{t}l(z)e^{\lambda z}dz+e^{-\lambda}C=\int_{0}^{t}l(s)e^{\lambda s}ds+C,
			\end{align}
			using that $l(z)=l(z+1)$ for all $z\in\mathbb{R}$ since $l\in C^{\infty}(S^1)$. By \eqref{eq:integrals}, we necessarily have that
			\[ C=\frac{e^{-\lambda}}{1-e^{-\lambda}}\int_{0}^{1}l(s)e^{\lambda s}ds.
			\]
		In conclusion, the unique solution $g\in C^{\infty}(S^1)$ to \eqref{eq:to-solve} is given by
		\[
		g(t)=e^{-\lambda t}\left(\int_{0}^{t}l(s)e^{\lambda s}ds+\frac{e^{-\lambda}}{1-e^{-\lambda}}\int_{0}^{1}l(s)e^{\lambda s}ds\right).\qedhere
		\]
	\end{proof}
	
	We can now show that for model Lie affine foliations $(\mathcal{G}_{\alpha,\lambda},\omega)$, the rigidity requirements of Cor.~\ref{cor:req-rigidity} are independent of the choice of Maurer-Cartan form $\omega$ defining $\mathcal{G}_{\alpha,\lambda}$.

	\begin{cor}\label{cor:equiv-rigidity}
		Consider a model Lie affine foliation $(\mathcal{G}_{\alpha,\lambda},\omega)$. Let $\zeta_1,\zeta_2\in\Omega^{1}(\mathbb{T}_{\varphi})$ be independent one-forms such that $T\mathcal{G}_{\alpha,\lambda}=\ker\zeta_1\cap\ker\zeta_2$ and
		\begin{equation*}
			\begin{cases}
				d\zeta_1-\zeta_2\wedge\zeta_1=0\\
				d\zeta_2=0
			\end{cases}.
		\end{equation*}
		Then $(\omega_1,\omega_2)$ satisfies the requirements of Cor.~\ref{cor:req-rigidity} if and only if $(\zeta_1,\zeta_2)$ does.
	\end{cor}
	\begin{proof}
		By Prop.~\ref{prop:other-forms}, we know that there exist $f,g\in C^{\infty}(S^{1})$ and $K\in\mathbb{R}_{0}$ such that
		\begin{equation*}
			\begin{cases}
				\zeta_2=\omega_2+df\\
				\zeta_1=e^{f}\left(K\omega_1+dg-g\omega_2\right)
			\end{cases}.
		\end{equation*}
		This means that under the isomorphism
		\begin{equation*}
			H^{\bullet}_{\zeta_2}(\mathbb{T}_\varphi)\rightarrow H^{\bullet}_{\omega_2}(\mathbb{T}_\varphi):[\eta]\mapsto [e^{-f}\eta],
		\end{equation*}
		we have that $[\zeta_1]\in H^{1}_{\zeta_2}(\mathbb{T}_\varphi)$ corresponds with $K[\omega_1]\in H^{1}_{\omega_2}(\mathbb{T}_\varphi)$. Hence, $H^{1}_{\omega_2}(\mathbb{T}_\varphi)=\mathbb{R}[\omega_1]$ if and only if $H^{1}_{\zeta_2}(\mathbb{T}_\varphi)=\mathbb{R}[\zeta_1]$. Next, note that we get a commutative diagram
		\begin{equation}\label{eq:diag3}
			\begin{tikzcd}[row sep=large, column sep=large]
				& & H^{2}_{\omega_2}(\mathbb{T}_\varphi)\\
				&H^{1}(\mathbb{T}_\varphi)\arrow[swap]{r}{\bullet\wedge[\zeta_1]}\arrow{ru}{\bullet\wedge K[\omega_1]} & H^{2}_{\zeta_2}(\mathbb{T}_\varphi)\arrow[swap]{u}{[\eta]\mapsto [e^{-f}\eta]}
			\end{tikzcd}.
		\end{equation}
		Indeed, following the bottom arrow and then the right arrow, we get
		\[
		H^{1}(\mathbb{T}_\varphi)\rightarrow H^{2}_{\omega_2}(\mathbb{T}_\varphi):[\gamma]\mapsto [\gamma\wedge K\omega_1]+[\gamma\wedge (dg-g\omega_2)].
		\]
		The second summand is trivial because
		\[
		d(g\gamma)-\omega_2\wedge g\gamma=-\gamma\wedge (dg-g\omega_2).
		\]
		So the diagram \eqref{eq:diag3} commutes. Because the right arrow is an isomorphism and $K\neq 0$, it follows that injectivity of the map $\bullet\wedge[\omega_1]:H^{1}(\mathbb{T}_\varphi)\rightarrow H^{2}_{\omega_2}(\mathbb{T}_\varphi)$ is equivalent with injectivity of the map $\bullet\wedge[\zeta_1]:H^{1}(\mathbb{T}_\varphi)\rightarrow H^{2}_{\zeta_2}(\mathbb{T}_\varphi)$. This finishes the proof.
	\end{proof}
	
	Consequently, to prove that a Lie affine foliation of the form $(\mathcal{G}_{\alpha,\lambda},\zeta)$ is rigid, it suffices to check that the model $(\mathcal{G}_{\alpha,\lambda},\omega)$ satisfies the requirements of Cor.~\ref{cor:req-rigidity}. This in turn can often be done in a straightforward way by checking the sufficient conditions stated in Cor.~\ref{cor:rigidity-model}.

	\subsubsection{Non-rigidity of model foliations}\label{subsub:non-rigid}
	We showed in \S\ref{subsub:rigidity-models} that a model Lie affine foliation $(\mathcal{G}_{\alpha,\lambda},\omega)$ can be rigid when deformed as a Lie foliation, namely if the conditions of Cor.~\ref{cor:rigidity-model} are satisfied. However, the foliation $\mathcal{G}_{\alpha,\lambda}$ on the mapping torus $\mathbb{T}_{\varphi}$ is never rigid when it is deformed merely as a foliation. We will now spell this out.
	
	\vspace{0.3cm}
	\noindent
	\underline{\emph{Background on deformation theory of foliations:}}
	Let $M$ be compact and $\mathcal{F}$ a $k$-dimensional foliation on $M$. In what follows, we will study smooth paths of foliations $\mathcal{F}_t$ with $\mathcal{F}_{0}=\mathcal{F}$. We say that $\mathcal{F}$ is rigid if for every such deformation $\mathcal{F}_t$ of $\mathcal{F}$, there is an isotopy $(\phi_t)$ of $M$ such that $T\mathcal{F}_t=(\phi_t)_{*}T\mathcal{F}$ for small $t$. We now introduce a suitable cohomology group whose vanishing can be interpreted heuristically as the infinitesimal requirement for rigidity.

	The normal bundle $N\mathcal{F}:=TM/T\mathcal{F}$ carries a flat $T\mathcal{F}$-connection $\nabla$, called the Bott connection, which is defined by
	\begin{equation}\label{eq:bott}
		\nabla_{X}\overline{Y}=\overline{[X,Y]},\hspace{1cm} X\in\Gamma(T\mathcal{F}), \overline{Y}\in\Gamma(N\mathcal{F}).
	\end{equation}
	We obtain a complex $\big(\Omega^{\bullet}(\mathcal{F};N\mathcal{F}),d_{\nabla}\big)$, where $\Omega^{k}(\mathcal{F};N\mathcal{F}):=\Gamma(\wedge^{k}T^{*}\mathcal{F}\otimes N\mathcal{F})$ are leafwise $k$-forms with coefficients in $N\mathcal{F}$, and the differential $d_{\nabla}$ is defined by
	\begin{align*}
		d_{\nabla}\eta(V_0,\ldots,V_k)&=\sum_{i=0}^{k}(-1)^{i}\nabla_{V_i}\big(\eta(V_0,\ldots,V_{i-1},\widehat{V_i},V_{i+1},\ldots,V_k)\big)\\
		&\hspace{0.5cm}+\sum_{i<j}(-1)^{i+j}\eta\big([V_i,V_j],V_0,\ldots,\widehat{V_i},\ldots,\widehat{V_j},\ldots,V_k\big).
	\end{align*}
	Heitsch \cite{heitsch} showed that the infinitesimal deformation of $\mathcal{F}$ induced by a smooth deformation $\mathcal{F}_t$ is a one-cocycle in  $\big(\Omega^{\bullet}(\mathcal{F};N\mathcal{F}),d_{\nabla}\big)$. Moreover, if the path $\mathcal{F}_t$ is obtained by applying an isotopy to $\mathcal{F}$, then the corresponding infinitesimal deformation is a one-coboundary. In more detail, let us identify $N\mathcal{F}$ with a subbundle of $TM$ complementary to $T\mathcal{F}$. Since $M$ is compact, the distribution $T\mathcal{F}_t$ is still transverse to $N\mathcal{F}$ for small enough $t$. Hence there exist $\eta_t\in\Gamma(T^{*}\mathcal{F}\otimes N\mathcal{F})$ such that
	\[
	T\mathcal{F}_t=\text{Graph}(\eta_t)=\{V+\eta_t(V):\ V\in\Gamma(T\mathcal{F})\}.
	\]
	Moreover, we need to require that the distribution $\text{Graph}(\eta_t)$ is involutive, so that it indeed integrates to a foliation $\mathcal{F}_t$. This gives a non-linear equation in $\eta_t$, with linearization
	\[
	d_{\nabla}\dt{\eta_0}=0.
	\]
	If the path $\mathcal{F}_t$ is generated by an isotopy $(\phi_t)$, i.e. $T\mathcal{F}_t=(\phi_t)_{*}T\mathcal{F}$, then the corresponding infinitesimal deformation is exact. Indeed,
	\[
	\dt{\eta_0}=d_{\nabla}(\overline{V_0}),
	\]
	where $(V_t)$ is the time-dependent vector field corresponding with the isotopy $(\phi_t)$. It follows that, loosely speaking, one can interpret the vanishing of the cohomology group $H^1(\mathcal{F},N\mathcal{F})$ as the infinitesimal requirement for rigidity of $\mathcal{F}$.

	\vspace{0.3cm}
	\noindent
	\underline{\emph{Model foliations are not rigid:}}
	Let us now consider a model Lie affine foliation $(\mathcal{G}_{\alpha,\lambda},\omega)$. We will show that the foliation $\mathcal{G}_{\alpha,\lambda}$ on $\mathbb{T}_{\varphi}$ is not rigid when deformed merely as a foliation.
	
	\begin{lemma}\label{lem:dim}
		The fiber $L$ of $\mathbb{T}_{\varphi}\rightarrow S^{1}$ necessarily satisfies $\dim H^{1}(L)\geq 2$.
	\end{lemma}
	\begin{proof}
		If we would have $\dim H^{1}(L)=1$, then necessarily the class $[\alpha]\in H^{1}(L)$ would be a real multiple of an integral cohomology class. Hence the restricted foliation $\mathcal{G}_{\alpha,\lambda}|_{L}$ on $L$ defined by $\alpha\in\Omega^{1}(L)$ would be given by the fibers of a fibration. This contradicts that the leaves of $\mathcal{G}_{\alpha,\lambda}|_{L}$ are dense in $L$, see Step 1 in the proof of Prop.~\ref{prop:other-forms}.
	\end{proof}
	
	\begin{prop}\label{prop:not-rigid}
		The foliation $\mathcal{G}_{\alpha,\lambda}$ on $\mathbb{T}_{\varphi}$ is not rigid.
	\end{prop}
	\begin{proof}
		We construct a smooth path of foliations $\mathcal{G}_s$ such that $\mathcal{G}_{0}=\mathcal{G}_{\alpha,\lambda}$ and the infinitesimal deformation coming from $\mathcal{G}_s$ defines a non-trivial class in $H^{1}(\mathcal{G}_{\alpha,\lambda},N\mathcal{G}_{\alpha,\lambda})$.
		
		By Lemma~\ref{lem:dim} we can take $[\beta]\in H^{1}(L)$ such that $[\beta]$ and $[\alpha]$ are linearly independent. Also fix a bump function $\rho(t)\in C^{\infty}[0,1]$ such that $\rho\equiv 0$ near $t=0$ and $t=1$. It follows that we get a well-defined one-form $\rho(t)\beta\in\Omega^{1}(\mathbb{T}_{\varphi})$. Now consider the path of foliations $\mathcal{G}_s$ defined by the kernel of the one-forms
		\[
		\begin{cases}
			\omega_1(s)=e^{-\lambda t}\alpha\\
			\omega_2(s)=-\lambda dt+s\rho(t)\beta	
		\end{cases}.
		\]
		Note that for every value of $s$, we have that $\omega_1(s)$ and $\omega_2(s)$ are linearly independent at every point of $\mathbb{T}_{\varphi}$. To see that $\omega_1(s)$ and $\omega_2(s)$ indeed define a foliation $\mathcal{G}_s$, we need to check that they generate a differential ideal. This is the case, because
		\[
		\begin{cases}
			d\omega_1(s)=-\lambda dt\wedge\omega_1(s)\\
			d\omega_2(s)=s\rho'(t)dt\wedge\beta=\lambda^{-1}s\rho'(t)\beta\wedge\omega_2(s)
		\end{cases},
		\]
		where $-\lambda dt$ and $\lambda^{-1}s\rho'(t)\beta$ are well-defined one-forms on $\mathbb{T}_{\varphi}$. If  $r:\Omega^{1}(\mathbb{T}_{\varphi})\rightarrow\Omega^{1}(\mathcal{G}_{\alpha,\lambda})$ denotes the restriction map, then we have
		\[
		T\mathcal{G}_{s}=\text{Graph}\left(s\lambda^{-1}r\left(\rho(t)\beta\right)\otimes\partial_t\right)=\left\{V+s\lambda^{-1}\rho(t)\beta(V)\partial_t:\ V\in T\mathcal{G}_{\alpha,\lambda}\right\},
		\]
		i.e. $T\mathcal{G}_{s}$ is defined by the element $s\lambda^{-1}r\left(\rho(t)\beta\right)\otimes\partial_{t}\in\Omega^{1}(\mathcal{G}_{\alpha,\lambda},N\mathcal{G}_{\alpha,\lambda})$. To conclude that the deformation $\mathcal{G}_{s}$ is not generated by an isotopy, we need to check that the following cohomology class is non-trivial:
		\begin{equation}\label{eq:class}
			\left[\lambda^{-1}r\left(\rho(t)\beta\right)\otimes\partial_{t}\right]\in H^{1}(\mathcal{G}_{\alpha,\lambda},N\mathcal{G}_{\alpha,\lambda}).
		\end{equation}
		
		To do this, note that $\partial_t$ defines a flat section of $(N\mathcal{G}_{\alpha,\lambda},\nabla)$ since its flow preserves $\mathcal{G}_{\alpha,\lambda}$. Using the Lie foliation structure of $\mathcal{G}_{\alpha,\lambda}$, we can find another flat section $Z\in\Gamma(N\mathcal{G}_{\alpha,\lambda})$ so that $\{\partial_t,Z\}$ is a flat frame of $(N\mathcal{G}_{\alpha,\lambda},\nabla)$. This is a consequence of the well-known fact that the frame of $N\mathcal{G}_{\alpha,\lambda}$ dual to the frame $\{e^{-\lambda t}\alpha,-\lambda dt\}$ of $N^{*}\mathcal{G}_{\alpha,\lambda}$ consists of flat sections \cite[Proof of Prop.~4.21]{Moerdijk}.
		 If the class \eqref{eq:class} were trivial, there would be $f,g\in C^{\infty}(\mathbb{T}_{\varphi})$ such that
		\[
		r\left(\rho(t)\beta\right)\otimes\partial_{t}=d_{\nabla}\left(f\partial_t+gZ\right)=d_{\mathcal{G}_{\alpha,\lambda}}f\otimes\partial_t+d_{\mathcal{G}_{\alpha,\lambda}}g\otimes Z,
		\]
		hence we would have that $r\left(\rho(t)\beta\right)=d_{\mathcal{G}_{\alpha,\lambda}}f$. Restricting to a fiber $L$ of $\mathbb{T}_{\varphi}\rightarrow S^{1}$, we get that the class $[\beta]\in H^{1}(L)$ lies in the kernel of the map
		\begin{equation}\label{eq:beta}
			H^{1}(L)\rightarrow H^{1}\left(\mathcal{G}_{\alpha,\lambda}|_{L}\right).
		\end{equation}
		Now recall that there is a short exact sequence of complexes
		\[
		0\longrightarrow\big(\Omega^{\bullet}_{\mathcal{G}_{\alpha,\lambda}|_{L}}(L),d\big)\longrightarrow\big(\Omega^{\bullet}(L),d\big)\longrightarrow\big(\Omega^{\bullet}(\mathcal{G}_{\alpha,\lambda}|_{L}),d_{\mathcal{G}_{\alpha,\lambda}|_{L}}\big)\longrightarrow 0,
		\]
		where $\Omega^{\bullet}_{\mathcal{G}_{\alpha,\lambda}|_{L}}(L)$ are the forms on $L$ that vanish when all arguments are tangent to $\mathcal{G}_{\alpha,\lambda}|_{L}$. The corresponding long exact cohomology sequence is
		\[
		\cdots\longrightarrow H^{1}_{\mathcal{G}_{\alpha,\lambda}|_{L}}(L)\longrightarrow H^{1}(L)\longrightarrow H^{1}\left(\mathcal{G}_{\alpha,\lambda}|_{L}\right)\longrightarrow \cdots
		\]
		Hence, $[\beta]$ lying in the kernel of the map \eqref{eq:beta} means that $[\beta]$ lies in the image of 
		\begin{equation}\label{eq:image}
			H^{1}_{\mathcal{G}_{\alpha,\lambda}|_{L}}(L)\longrightarrow H^{1}(L).
		\end{equation}
		However, we claim that $H^{1}_{\mathcal{G}_{\alpha,\lambda}|_{L}}(L)=\mathbb{R}\alpha$. Indeed, a one-form on $L$ that vanishes on $\mathcal{G}_{\alpha,\lambda}|_{L}$ must be of the form $h\alpha$ for some $h\in C^{\infty}(L)$. Such a form is closed if $dh\wedge\alpha=0$, meaning that $h$ is a basic function for the foliation $\mathcal{G}_{\alpha,\lambda}|_{L}$ defined by $\alpha$. But since $\mathcal{G}_{\alpha,\lambda}|_{L}$ has dense leaves (see Step 1 in the proof of Prop.~\ref{prop:other-forms}), this means that $h$ is constant. It follows that the image of the map \eqref{eq:image} is $\mathbb{R}[\alpha]$. So we get that $[\beta]\in\mathbb{R}[\alpha]$, which is impossible because $[\beta]$ and $[\alpha]$ are linearly independent. This confirms that the class \eqref{eq:class} is non-trivial.
	\end{proof}

	\subsection{Rigidity of Lie affine foliations in low dimension}
	We now restrict attention to Lie affine foliations $(\mathcal{F},\omega)$ on compact, connected, orientable manifolds $M$ of dimension $3$ or $4$. As shown in \cite{caron} and \cite{Matsumoto}, such foliated manifolds $(M,\mathcal{F})$ are foliated diffeomorphic to a foliated manifold $(\mathbb{T}_{\varphi},\mathcal{G}_{\alpha,\lambda})$ as in Def.~\ref{def:model}. Hence, the foliation $\mathcal{F}$ is never rigid when deformed merely as a foliation, because of Prop.~\ref{prop:not-rigid}. However, we will now show that the Lie affine foliation $(\mathcal{F},\omega)$ \emph{is} rigid when deformed as a Lie foliation.

	\subsubsection{Rigidity of Lie affine foliations on $3$-manifolds}
	
	Let $M$ be a compact, connected and orientable $3$-manifold. There is a normal form for foliations on $M$ with a Lie affine structure due to Caron \cite{caron}. We start by recalling this result.
	
	Let $A$ be a matrix in $SL_{2}(\mathbb{Z})$ with $\text{Tr}(A)>2$. Then $A$ has two real eigenvalues $\lambda,\lambda^{-1}\in\mathbb{R}$ that are positive and distinct. Pick eigenvectors $V,W$ associated to the eigenvalues $\lambda,\lambda^{-1}$ respectively. We will view $A$ as a diffeomorphism $\varphi_A$ of the torus $\mathbb{T}^{2}$ and $V,W$ as vector fields on $\mathbb{T}^{2}$. Now consider the mapping torus
	\begin{equation}\label{eq:mapping-A}
		\mathbb{T}_{A}:=\frac{\mathbb{T}^{2}\times\mathbb{R}}{(x,t)\sim(\varphi_A(x),t+1)}.
	\end{equation}
	The distributions $\text{Span}(V)$ and $\text{Span}(W)$ on $\mathbb{T}^{2}$ are invariant under $\varphi_A$, hence they induce $1$-dimensional foliations $\mathcal{F}_V$ and $\mathcal{F}_W$ on $\mathbb{T}_{A}$. Both are irrational flows tangent to the $\mathbb{T}^{2}$-fibers of $\mathbb{T}_{A}\rightarrow S^{1}$. In particular, the leaf closures of $\mathcal{F}_V$ and $\mathcal{F}_W$ are the fibers of $\mathbb{T}_{A}\rightarrow S^{1}$.
	
	Both $\mathcal{F}_V$ and $\mathcal{F}_W$ can be endowed with Lie $\mathfrak{aff}(1)$-foliation structures. Let us explain this in the case of $\mathcal{F}_V$. Denote by $V^{*},W^{*}\in\Omega^{1}(\mathbb{T}^{2})$ the dual frame of $V,W\in\mathfrak{X}(\mathbb{T}^{2})$ -- note that both $V^{*}$ and $W^{*}$ are closed. The one-forms $\lambda^{t}W^{*}$ and $\ln(\lambda)dt$ on $\mathbb{T}^{2}\times\mathbb{R}$ are invariant under the identification appearing in \eqref{eq:mapping-A}, hence they descend to one-forms on $\mathbb{T}_A$ which we will denote in the same way
	\begin{equation}\label{eq:omegas}
		\begin{cases}
			\omega_1:=\lambda^{t}W^{*}\\
			\omega_2:=\ln(\lambda)dt
		\end{cases}.
	\end{equation}
	Clearly, the foliation $\mathcal{F}_V$ is defined by $\omega_1$ and $\omega_2$, which satisfy the relations \eqref{eq:aff-fol}. Hence, we get a Lie $\mathfrak{aff}(1)$-foliation  $(\mathcal{F}_V,\omega)$. Note that $(\mathcal{F}_V,\omega)$ is of the type introduced in Def.~\ref{def:model} -- in the notation established there, we have $(\mathcal{F}_V,\omega)=(\mathcal{G}_{W^{*},-\ln(\lambda)},\omega)$.
	
	\begin{thm}\cite[Thm.~V.1]{caron}\label{thm:foliated-diffeo}
		Let $M$ be a compact, connected and orientable $3$-manifold with a Lie $\mathfrak{aff}(1)$-foliation $(\mathcal{F},\eta)$. Then there exist a matrix $A\in SL_{2}(\mathbb{Z})$ with $\text{Tr}(A)>2$ and a foliated diffeomorphism
		$(\mathbb{T}_A,\mathcal{F}_V)\overset{\sim}{\rightarrow}(M,\mathcal{F})$.
	\end{thm}
	
	%This normal form will allow us to prove that any Lie $\mathfrak{aff}(1)$-foliation $\mathcal{F}$ on a compact, connected and orientable $3$-manifold is rigid. We first need to check that the model $(\mathbb{T}_A,\mathcal{F}_V)$ satisfies the assumptions of Cor.~\ref{cor:req-rigidity}.
	\begin{cor}\label{cor:rigid3}
		Let $M$ be a compact, connected and orientable $3$-manifold.  Any Lie $\mathfrak{aff}(1)$-foliation $(\mathcal{F},\eta)$ on $M$ is rigid, when deformed as a Lie foliation.
	\end{cor}
	\begin{proof}
		By assumption, $\mathcal{F}$ is defined by one-forms $\eta_1,\eta_2\in\Omega^{1}(M)$ satisfying
		\begin{equation*}
			\begin{cases}
				d\eta_1-\eta_2\wedge\eta_1=0\\
				d\eta_2=0
			\end{cases}.
		\end{equation*}
		By Thm.~\ref{thm:foliated-diffeo}, there is a matrix $A\in SL_{2}(\mathbb{Z})$ with $\text{Tr}(A)>2$ and a foliated diffeomorphism
		\[
		\phi:(\mathbb{T}_A,\mathcal{F}_V)\overset{\sim}{\rightarrow}(M,\mathcal{F}).
		\]
		We denote $\zeta_1:=\phi^{*}\eta_1$ and $\zeta_2:=\phi^{*}\eta_2$, which are one-forms on $\mathbb{T}_A$ that define $\mathcal{F}_V$. They also satisfy the relations
		\begin{equation*}
			\begin{cases}
				d\zeta_1-\zeta_2\wedge\zeta_1=0\\
				d\zeta_2=0
			\end{cases}.
		\end{equation*}
		To show that $(\mathcal{F},\eta)$ is rigid, it suffices to establish rigidity of $(\mathcal{F}_V,\zeta)$. To do so, we have to check that $(\mathcal{F}_V,\zeta)$ meets the requirements of Cor.~\ref{cor:req-rigidity}. By Cor.~\ref{cor:equiv-rigidity}, we can check instead that $(\mathcal{F}_V,\omega)$ meets these requirements. 
		This in turn can be done by checking that the sufficient conditions in Cor.~\ref{cor:rigidity-model} are satisfied. So we need to verify the following:
		\begin{enumerate}
			\item The eigenvalue $\lambda^{-1}$ of $[\varphi_{A}^{*}]:H^{1}(\mathbb{T}^{2})\rightarrow H^{1}(\mathbb{T}^{2})$ has algebraic multiplicity $1$,
			\item The map $\bullet\wedge[W^{*}]:\ker\left([\varphi^{*}_{A}-\mathrm{Id}]:H^{1}(\mathbb{T}^{2})\rightarrow H^{1}(\mathbb{T}^{2})\right)\rightarrow H^{2}(\mathbb{T}^{2})$ is injective.
		\end{enumerate}
		Both conditions are met, since $[\varphi_{A}^{*}]:H^{1}(\mathbb{T}^{2})\rightarrow H^{1}(\mathbb{T}^{2})$ has distinct eigenvalues $\lambda,\lambda^{-1}$.
	\end{proof}
	
	\begin{ex}\label{ex:dim3}
	Consider the torus $\mathbb{T}^{2}$ with the diffeomorphism $\varphi_{A}:\mathbb{T}^{2}\rightarrow\mathbb{T}^{2}$ induced by 
	\[
	A=\begin{pmatrix}
	2 & 1 \\ 1 & 1
	\end{pmatrix}\in SL_{2}(\mathbb{Z}).
	\]
	We remark that $\varphi_{A}$ is also known as Arnold's cat map. The eigenvalues of $A$ are given by
	\[
	\lambda_{1}=\frac{3+\sqrt{5}}{2},\hspace{0.5cm}\lambda_{2}=\frac{3-\sqrt{5}}{2}.
	\]
	Consider the closed, nowhere vanishing one-form $\alpha\in\Omega^{1}(\mathbb{T}^{2})$ defined by
	\[
	\alpha:=\left(\frac{1+\sqrt{5}}{2}\right)d\theta_1+d\theta_2.
	\]
	It is straightforward to check that $\varphi_{A}^{*}\alpha=\lambda_1\alpha$. Therefore, the mapping torus $\mathbb{T}_{A}$ supports a model Lie affine foliation $(\mathcal{G}_{\alpha,\ln(\lambda_1)},\omega)$ defined by the one-forms
	\[
	\begin{cases}
			\omega_1:=\lambda_{1}^{-t}\alpha\\
			\omega_2=-\ln(\lambda_1)dt
	\end{cases}.
	\]
	The Lie $\mathfrak{aff}(1)$-foliation $(\mathcal{G}_{\alpha,\ln(\lambda_1)},\omega)$ is rigid. Indeed, the conditions of Cor.~\ref{cor:rigidity-model} are satisfied:
	\begin{enumerate}
		\item The eigenvalue $\lambda_1$ of $[\varphi_{A}^{*}]:H^{1}(\mathbb{T}^{2})\rightarrow H^{1}(\mathbb{T}^{2})$ has algebraic multiplicity $1$,
		\item The map $\bullet\wedge[\alpha]:\ker\left([\varphi^{*}_{A}-\mathrm{Id}]:H^{1}(\mathbb{T}^{2})\rightarrow H^{1}(\mathbb{T}^{2})\right)\rightarrow H^{2}(\mathbb{T}^{2})$ is injective.
	\end{enumerate}
	Both statements hold because the matrix $A$ has distinct eigenvalues $\lambda_1,\lambda_2$ different from $1$.
	\end{ex}

	\color{black}

	Although $(\mathcal{F},\eta)$ is rigid when deformed as a Lie foliation, we already know from Prop.~\ref{prop:not-rigid} that the foliation $\mathcal{F}$ is not rigid when deformed just as a foliation. In the proof of Prop.~\ref{prop:not-rigid} we constructed an explicit non-trivial deformation of $\mathcal{F}_{V}$ on $\mathbb{T}_{A}$, which under the foliated diffeomorphism $(\mathbb{T}_{A},\mathcal{F}_{V})\overset{\sim}{\rightarrow}(M,\mathcal{F})$ corresponds with a non-trivial deformation of $\mathcal{F}$. In fact, when $M$ is $3$-dimensional we can show that $\mathcal{F}$ has plenty of non-trivial deformations since the relevant cohomology group $H^1(\mathcal{F},N\mathcal{F})$ can be computed. We now spell this out.

	\begin{remark}
		Let $(\mathcal{F},\eta)$ be a Lie affine foliation on a compact, connected, orientable $3$-manifold $M$.
		As we explained in \S\ref{subsub:non-rigid}, the infinitesimal deformation induced by a smooth deformation $\mathcal{F}_t$ of $\mathcal{F}$ is a one-cocycle in $\left(\Omega^{\bullet}(\mathcal{F},N\mathcal{F}),d_{\nabla}\right)$. Since $\mathcal{F}$ is one-dimensional, every one-cocycle $\eta\in\Omega^{1}(\mathcal{F},N\mathcal{F})$ arises in such a way. Indeed, since involutivity is automatic for distributions of rank $1$, we have a path of foliations $\mathcal{F}_t$ given by $T\mathcal{F}_t=\text{Graph}(t\eta)$. Its corresponding infinitesimal deformation is exactly $\eta$. Hence, every cocycle defining a non-zero class in $H^{1}(\mathcal{F},N\mathcal{F})$ gives rise to a non-trivial deformation of $\mathcal{F}$.

		%Now assume moreover that the foliation $\mathcal{F}$ is one-dimensional. In that case, every one-cocycle $\eta\in\Omega^{1}(\mathcal{F},N\mathcal{F})$ arises as the infinitesimal deformation coming from a path $\mathcal{F}_t$. That is to say, the deformation problem of a one-dimensional foliation $\mathcal{F}$ is unobstructed. Indeed, since involutivity is automatic for distributions of rank $1$, we have a path of foliations $\mathcal{F}_t$ given by $T\mathcal{F}_t=\text{Graph}(t\eta)$. Its corresponding infinitesimal deformation is exactly $\eta$. It follows that, for a one-dimensional foliation $\mathcal{F}$ to be rigid, a necessary condition is that $$H^{1}(\mathcal{F},N\mathcal{F})=0.$$
		
		We now proceed by computing $H^{1}(\mathcal{F},N\mathcal{F})$. We can assume that $M=\mathbb{T}_A$ and $\mathcal{F}=\mathcal{F}_V$ by Thm.~\ref{thm:foliated-diffeo}. To compute the cohomology group $H^{1}(\mathcal{F}_V,N\mathcal{F}_V)$, note that the normal bundle $N\mathcal{F}_V$ is trivial as a $T\mathcal{F}_V$-representation. Indeed, the frame $\overline{\partial_t},\overline{\lambda^{-t}W}$ of $N\mathcal{F}_V$ consists of sections that are flat with respect to the Bott connection \eqref{eq:bott}. It follows that
		\[
		H^{1}(\mathcal{F}_V,N\mathcal{F}_V)\cong H^{1}(\mathcal{F}_V)\otimes\mathbb{R}^{2}.
		\]
		The foliated cohomology group $H^{1}(\mathcal{F}_V)$ can be computed using results from \cite[\S 2]{diophantine}. There one considers manifolds $M$ with certain one-dimensional foliations generated by a nowhere vanishing vector field $X\in\mathfrak{X}(M)$. It is assumed that the closures of the orbits are tori $\mathbb{T}^{n}$, which are the fibers of a locally trivial fibration 
		\[
		\mathbb{T}^{n}\hookrightarrow M\rightarrow W.
		\]
		Moreover, on each fiber $\mathbb{T}^{n}$, the vector field $X$ is assumed to satisfy a Diophantine condition \cite[Def.~2.2]{diophantine}. The foliated manifold $(\mathbb{T}_A,\mathcal{F}_V)$ satisfies all of these requirements, see \cite[\S 5]{diophantine}. Hence, we can apply \cite[Cor.~2.5]{diophantine}, which tells us that
		$
		H^{1}(\mathcal{F}_V)\cong C^{\infty}(S^1).
	    $
		It follows that
		\[
		H^{1}(\mathcal{F}_V,N\mathcal{F}_V)\cong C^{\infty}(S^1)\otimes\mathbb{R}^{2}.
		\]
		So we obtain an infinite dimensional space. In conclusion, when $\mathcal{F}$ is deformed merely as a foliation, it has many deformations $\mathcal{F}_t$ that are not generated by an isotopy of $M$. 
	\end{remark}
	
	\subsubsection{Rigidity of Lie affine foliations on $4$-manifolds}
	Let $M$ be a compact, connected and orientable $4$-manifold. There is a normal form for foliations on $M$ with a Lie affine structure due to Matsumoto-Tsuchiya \cite{Matsumoto}. We start by recalling this result.
	
	Let $G$ be a simply connected $3$-dimensional nilpotent Lie group, i.e. either $\mathbb{R}^{3}$ or the Heisenberg group $H_3$. Let $\Gamma\subset G$ be a uniform lattice, so $N:=G/\Gamma$ is a compact manifold. Let $\varphi$ be an automorphism of $G$ which keeps the lattice $\Gamma$ invariant -- it induces a diffeomorphism of $N$ which we still denote by $\varphi$. Let $\alpha\in\Omega^{1}(N)$ be a closed non-zero left invariant one-form such that $\varphi^{*}\alpha=e^{\lambda}\alpha$ for some $\lambda\neq 0$. The mapping torus $\mathbb{T}_{\varphi}$ then carries a model Lie affine foliation $(\mathcal{G}_{\alpha,\lambda},\omega)$ defined by
	\begin{equation*}
		\begin{cases}
			\omega_1:=e^{-\lambda t}\alpha\\
			\omega_2=-\lambda dt
		\end{cases}.
	\end{equation*}
	The following possibilities can occur:
	\begin{enumerate}
		\item Either $\varphi$ is an Anosov automorphism of $N=\mathbb{T}^{3}$.
		\item Or $\varphi$ is the lift of an Anosov diffeomorphism under an $S^{1}$-bundle map $p:N\rightarrow\mathbb{T}^{2}$.
	\end{enumerate}
	
	\begin{thm}\cite[Thm.~2.6]{Matsumoto}\label{thm:matsumoto}
		Let $M$ be a compact, connected and orientable $4$-manifold with a Lie $\mathfrak{aff}(1)$-foliation $(\mathcal{F},\eta)$. Then there is a foliated diffeomorphism between $(M,\mathcal{F})$ and a model foliation $(\mathbb{T}_{\varphi},\mathcal{G}_{\alpha,\lambda})$ of the type just described.
	\end{thm}
	
	\begin{cor}\label{cor:rigid4}
		Let $M$ be a compact, connected and orientable $4$-manifold.  Any Lie $\mathfrak{aff}(1)$-foliation $(\mathcal{F},\eta)$ on $M$ is rigid, when deformed as a Lie foliation.
	\end{cor}
	\begin{proof}
	By Thm.~\ref{thm:matsumoto}, we know that $(M,\mathcal{F})$ is foliated diffeomorphic with a model $(\mathbb{T}_{\varphi},\mathcal{G}_{\alpha,\lambda})$, where the fiber $N$ of $\mathbb{T}_{\varphi}$ is either $N=\mathbb{T}^{3}$ or $N=H_{3}/\Gamma$.
	As in the proof of Cor.~\ref{cor:rigid3}, we only need to check that the Lie affine foliation $(\mathcal{G}_{\alpha,\lambda},\omega)$ satisfies the conditions of Cor.~\ref{cor:rigidity-model}:
	\begin{enumerate}[i)]
		\item The eigenvalue $e^{\lambda}$ of $[\varphi^{*}]:H^{1}(N)\rightarrow H^{1}(N)$ has algebraic multiplicity equal to $1$,
		\item $\bullet\wedge[\alpha]:\ker\left([\varphi^{*}-\mathrm{Id}]:H^{1}(N)\rightarrow H^{1}(N)\right)\rightarrow H^{2}(N)$ is injective.
	\end{enumerate}
	%We first make a preliminary observation regarding condition $i)$. Let $b_{1}(N)$ denote the first Betti number of $N$. Note that it is equal to $2$ or $3$. Indeed, either $N=\mathbb{T}^{3}$, in which case $b_{1}(N)=3$. Or $N=H_{3}/\Gamma$, where $H_{3}$ is the Heisenberg group. In that case, Nomizu \cite{Nomizu} showed that
	%$
	%H^{1}(N)\cong H^{1}(\mathfrak{h})=\mathfrak{h}/[\mathfrak{h},\mathfrak{h}],
	%$
	%where $\mathfrak{h}$ is the Heisenberg Lie algebra. This shows that $b_{1}(N)=2$ in that case.
	%Now recall that the dimension of the generalized $e^{\lambda}$-eigenspace
	%\[
	%\ker\left([\varphi^{*}-e^{\lambda}\mathrm{Id}]^{b_1(N)}\right)
	%\]
	%coincides with the algebraic multiplicity of $e^{\lambda}$ as an eigenvalue of $[\varphi^{*}]:H^{1}(N)\rightarrow H^{1}(N)$. Since $b_{1}(N)$ is $2$ or $3$, it follows that condition $i)$ above is satisfied as soon as the eigenvalue $e^{\lambda}$ has algebraic multiplicity equal to $1$. So we only need to show that the following hold:
	%\begin{enumerate}
		%\item The eigenvalue $e^{\lambda}$ of $[\varphi^{*}]:H^{1}(N)\rightarrow H^{1}(N)$ has algebraic multiplicity equal to $1$,
		%\item $\bullet\wedge[\alpha]:\ker\left([\varphi^{*}-\mathrm{Id}]:H^{1}(N)\rightarrow H^{1}(N)\right)\rightarrow H^{2}(N)$ is injective.
	%\end{enumerate}
	
	%\vspace{0.1cm}
	\noindent
	\underline{First case}: $N=\mathbb{T}^{3}$.
	
	\vspace{0.1cm}
	The map induced by $\varphi$ in homology is then represented by $A\in GL(3,\mathbb{Z})=\text{Aut}\left(H_{1}(\mathbb{T}^{3},\mathbb{Z})\right)$. It follows that the map $[\varphi^{*}]:H^{1}(\mathbb{T}^{3})\rightarrow H^{1}(\mathbb{T}^{3})$ is represented by the matrix $A^{T}\in GL(3,\mathbb{Z})$. 
	
	Let us first check that condition $i)$ is satisfied. Assume by contradiction that $e^{\lambda}$ would be a multiple root of the characteristic polynomial $\text{char}(A^{T})$. This then implies that $\text{char}(A^{T})$ is reducible over $\mathbb{Q}$, see \cite[Chapter~V, \S~6, Prop.~6.1]{Lang}. Since $\text{char}(A^{T})$ is a polynomial of degree $3$, it follows that $A^{T}$ has a rational eigenvalue $m/n\in\mathbb{Q}$. By the rational root theorem, $m$ would be a divisor of the constant term of $\text{char}(A^{T})$ and $n$ would be a divisor of the leading coefficient of $\text{char}(A^{T})$. It follows that necessarily $m/n=\pm1$. So we would get that the eigenvalues of $A^{T}$ are given by $e^{\lambda},e^{\lambda},\pm1$. Since $\lambda\neq 0$, this is incompatible with the fact that $|\det(A^{T})|=1$. This confirms that the condition $i)$ above holds.
	
	Regarding condition $ii)$, note that the kernel of the map $\bullet\wedge[\alpha]:H^{1}(\mathbb{T}^{3})\rightarrow H^{2}(\mathbb{T}^{3})$ is $\mathbb{R}[\alpha]$. The latter is contained in $\ker\left([\varphi^{*}-e^{\lambda}\mathrm{Id}]:H^{1}(\mathbb{T}^{3})\rightarrow H^{1}(\mathbb{T}^{3})\right)$, which has trivial intersection with $\ker\left([\varphi^{*}-\mathrm{Id}]:H^{1}(\mathbb{T}^{3})\rightarrow H^{1}(\mathbb{T}^{3})\right)$ because $\lambda\neq 0$. So condition $ii)$ holds.
	
	\vspace{0.1cm}
	\noindent
	\underline{Second case}: $N=H_{3}/\Gamma$.
	
	\vspace{0.1cm}
	
	In this case, the manifold $N$ is the total space of an $S^{1}$-bundle $p:N\rightarrow\mathbb{T}^{2}$ and $\varphi\in\text{Diff}(N)$ covers an Anosov diffeomorphism $\overline{\varphi}$ of $\mathbb{T}^{2}$. A result by Nomizu \cite{Nomizu} states that
	\[
	H^{1}(N)\cong H^{1}(\mathfrak{h})=\mathfrak{h}/[\mathfrak{h},\mathfrak{h}],
	\]
	where $\mathfrak{h}$ is the Heisenberg Lie algebra. This shows that $H^{1}(N)$ is two-dimensional, which implies that the pullback $[p^{*}]:H^{1}(\mathbb{T}^{2})\overset{\sim}{\rightarrow} H^{1}(N)$ is an isomorphism. It follows that the map $[\varphi^{*}]:H^{1}(N)\rightarrow H^{1}(N)$ can be identified with $[\overline{\varphi}^{*}]:H^{1}(\mathbb{T}^{2})\rightarrow H^{1}(\mathbb{T}^{2})$. The latter is represented by a hyperbolic matrix $A\in GL(2,\mathbb{Z})$, see \cite[\S2]{Palis}. Since $|\det A|=1$ and $1$ is not an eigenvalue of $A$, it is clear that conditions $i)$ and $ii)$ above are satisfied.
	\end{proof}

	\begin{ex}
	Consider the torus $\mathbb{T}^{3}$ with the diffeomorphism $\varphi_{A}:\mathbb{T}^{3}\rightarrow\mathbb{T}^{3}$ induced by
	\[
	A=\begin{pmatrix}
		2 & 1 & 0 \\ 1 & 1 & 0 \\ 0 & 0 & 1
	\end{pmatrix}\in SL_{3}(\mathbb{Z}).
	\]
	Note that we are in case $(2)$ described before Thm.~\ref{thm:matsumoto}. The eigenvalues of $A$ are given by
	\[
	\lambda_{1}=\frac{3+\sqrt{5}}{2},\hspace{0.5cm}\lambda_{2}=\frac{3-\sqrt{5}}{2},\hspace{0.5cm}\lambda_{3}=1.
	\]
	As in Ex.~\ref{ex:dim3}, we consider the closed, nowhere vanishing one-form $\alpha\in\Omega^{1}(\mathbb{T}^{2})$ given by
	\[
	\alpha:=\left(\frac{1+\sqrt{5}}{2}\right)d\theta_1+d\theta_2.
	\]
	It is straightforward to check that $\varphi_{A}^{*}\alpha=\lambda_1\alpha$. Therefore, the mapping torus $\mathbb{T}_{A}$ supports a model Lie affine foliation $(\mathcal{G}_{\alpha,\ln(\lambda_1)},\omega)$ defined by the one-forms
	\[
	\begin{cases}
		\omega_1:=\lambda_{1}^{-t}\alpha\\
		\omega_2=-\ln(\lambda_1)dt
	\end{cases}.
	\]
	The Lie $\mathfrak{aff}(1)$-foliation $(\mathcal{G}_{\alpha,\ln(\lambda_1)},\omega)$ is rigid. Indeed, the conditions of Cor.~\ref{cor:rigidity-model} are satisfied:
	\begin{enumerate}
		\item The eigenvalue $\lambda_1$ of $[\varphi_{A}^{*}]:H^{1}(\mathbb{T}^{3})\rightarrow H^{1}(\mathbb{T}^{3})$ has algebraic multiplicity $1$,
		\item The map $\bullet\wedge[\alpha]:\ker\left([\varphi^{*}_{A}-\mathrm{Id}]:H^{1}(\mathbb{T}^{3})\rightarrow H^{1}(\mathbb{T}^{3})\right)\rightarrow H^{2}(\mathbb{T}^{3})$ is injective.
	\end{enumerate}
	Condition $(1)$ holds since $A$ has three distinct eigenvalues. The map in condition $(2)$ reads 
	\[
	\bullet\wedge[\alpha]:\mathbb{R}[d\theta_{3}]\rightarrow H^{2}(\mathbb{T}^{2}):c[d\theta_{3}]\mapsto c\left(\frac{1+\sqrt{5}}{2}\right)[d\theta_3\wedge d\theta_1]+c[d\theta_3\wedge d\theta_2],
	\]
	which is clearly injective. This shows that also condition $(2)$ holds, so $(\mathcal{G}_{\alpha,\ln(\lambda_1)},\omega)$ is rigid.
	\end{ex}

	\color{black}
	\subsubsection{Higher dimensional manifolds}
	Lie affine foliations are not necessarily rigid when the underlying compact, connected and orientable manifold $M$ is of dimension $n\geq 5$. We now illustrate this by means of a concrete example.
	Consider the matrix $A\in SL_{4}(\mathbb{Z})$ given by
	\[
	A=\begin{pmatrix}
		2 & 1 & 0 & 0\\
		1 & 1 & 0 & 0\\
		0 & 0 & 2 & 1\\
		0 & 0 & 1 & 1
	\end{pmatrix}.
	\]
	It has two positive, distinct eigenvalues $\lambda,\lambda^{-1}\in\mathbb{R}$ of multiplicity $2$, and four linearly independent eigenvectors of the form $(V_1,0),(0,V_1),(V_2,0),(0,V_2)$ for $V_1,V_2\in\mathbb{R}^{2}$. It follows that there exist two linearly independent, closed, left invariant one-forms $\alpha_1,\alpha_2\in\Omega^{1}(\mathbb{T}^{4})$ whose pullback under the diffeomorphism $\varphi_A$ of the torus $\mathbb{T}^{4}$ satisfies
	\[
	\begin{cases}
		\varphi_{A}^{*}\alpha_1=\lambda\alpha_1\\
		\varphi_{A}^{*}\alpha_2=\lambda\alpha_2
	\end{cases}.
	\]
	On the mapping torus $\mathbb{T}_{A}$,
	we get a model Lie $\mathfrak{aff}(1)$-foliation $(\mathcal{F},\omega)$ given by the one-forms
	\[
	\begin{cases}
		\omega_1:=\lambda^{-t}\alpha_1\\
		\omega_2:=-\ln(\lambda)dt
	\end{cases}.
	\]
	The deformation cohomology of the Lie foliation $(\mathcal{F},\omega)$ is non-zero. Using the isomorphism
	\begin{equation}\label{eq:iso-last}
	H^{1}_{-\ln(\lambda)dt}(\mathbb{T}_{A})\overset{\sim}{\longrightarrow}\ker\left([\varphi_{A}^{*}-\lambda\mathrm{Id}]:H^{1}(\mathbb{T}^{4})\rightarrow H^{1}(\mathbb{T}^{4})\right)
	\end{equation}
	from Cor.~\ref{cor:computed} $ii)$, we get $H^{1}_{-\ln(\lambda)dt}(\mathbb{T}_{A})=\mathbb{R}[\lambda^{-t}\alpha_1]\oplus\mathbb{R}[\lambda^{-t}\alpha_2]$. This prevents the deformation cohomology from vanishing, which requires $H^{1}_{-\ln(\lambda)dt}(\mathbb{T}_{A})=\mathbb{R}[\lambda^{-t}\alpha_1]$ by Cor.~\ref{cor:injective}. 
	
The extra class $[\lambda^{-t}\alpha_2]\in H^{1}_{-\ln(\lambda)dt}(\mathbb{T}_{A})$ does indeed give rise to a non-trivial deformation of the Lie foliation $(\mathcal{F},\omega)$. Consider the path of Lie $\mathfrak{aff}(1)$-foliations $(\mathcal{F}_s,\omega_s)$ defined by 
	\[
	\begin{cases}
		\omega_1(s):=\lambda^{-t}\left(\alpha_1+s\alpha_2\right)\\
		\omega_2(s):=-\ln(\lambda)dt
	\end{cases}.
	\]
	Note that $(\mathcal{F}_0,\omega_0)=(\mathcal{F},\omega)$. In the language of \S\ref{subsec:def-lie}, the deformation $(\mathcal{F}_s,\omega_s)$ is specified by a path $\sigma(s)$ in $\Omega^{1}(\mathbb{T}_{A})\times\Omega^{1}(\mathbb{T}_{A})$ and a path $\psi(s)$ in $\wedge^{2}\mathfrak{aff}(1)^{*}\otimes\mathfrak{aff}(1)$, which read
	\[
	\begin{cases}
		\sigma(s)=\left(s\lambda^{-t}\alpha_2,0\right)\\
		\psi(s)=0
	\end{cases}.
	\]
	If $(\mathcal{F}_s,\omega_s)$ would be a trivial deformation of $(\mathcal{F},\omega)$, then the infinitesimal deformation
	$$
	\big(\dt{\sigma(0)},\dt{\psi}(0)\big)=\left(\lambda^{-t}\alpha_2,0\right)
	$$ would be exact in the deformation complex $\big(\mathcal{A}^{\bullet},D\big)$. By Prop.~\ref{prop:coho}, this would mean that there exist $C\in\mathbb{R}$ and $g\in C^{\infty}(\mathbb{T}_{A})$ such that
	\[
	\lambda^{-t}\alpha_2=C\lambda^{-t}\alpha_1-dg-g\ln(\lambda)dt.
	\]
	In other words, we have that $[\lambda^{-t}\alpha_2]=C[\lambda^{-t}\alpha_1]$ in $H^{1}_{-\ln(\lambda)dt}(\mathbb{T}_{A})$. This is a contradiction. So the path $(\mathcal{F}_{s},\omega_{s})$ is a non-trivial deformation of $(\mathcal{F},\omega)$. In particular, $(\mathcal{F},\omega)$ is not rigid.


\begin{thebibliography}{[1]}
		\bibitem{Banyaga} A. Banyaga, \textit{Examples of non-$d_{\omega}$-exact locally conformal symplectic forms}, J. Geom. \textbf{87}, p. 1-13, 2007.
		\bibitem{bazzoni} G. Bazzoni, M. Fern\'{a}ndez and V. Mu\~{n}oz, \textit{Non-formal co-symplectic manifolds}, Trans. Amer. Math. Soc. \textbf{367}(6), p. 4459-4481, 2015.
		\bibitem{Bunke} U. Bunke, \textit{Vorlesung Algebraische Topologie I}, University of Regensburg, Lecture notes.
		Available at \url{https://bunke.app.uni-regensburg.de/topos.pdf}.
		\bibitem{caron} P. Caron, \textit{Flots transversalement de Lie}, Ph.D. thesis, Universit\'{e} des sciences et techniques de Lille I, 1980. Available at \url{https://pepite-depot.univ-lille.fr/LIBRE/Th_Num/1980/50376-1980-112.pdf}.
		\bibitem{chen} X. Chen, \textit{Morse-Novikov cohomology of almost nonnegatively curved manifolds}, Adv. Math. \textbf{371}, Article 107249, 2020.
		\bibitem{conlon} L. Conlon, \textit{Differentiable Manifolds: A First Course}, Birkh\"{a}user Advanced Texts, Birkh\"{a}user Boston, 1993.	
		\bibitem{diophantine} A. Dehghan-Nezhad and A. El Kacimi Alaoui, \textit{\'{E}quations cohomologiques de flots Riemanniens et de diff\'{e}omorphismes d'Anosov}, J. Math. Soc. Japan \textbf{59}(4), p. 1105–1134, 2007.
		\bibitem{Leon} M. De Le\'{o}n, B. L\'{o}pez, J.C. Marrero and E. Padr\'{o}n, \textit{On the computation of the Lichnerowicz-Jacobi cohomology}, J. Geom. Phys. \textbf{44}(4), p. 507-522, 2003.
		\bibitem{Lie} A. El Kacimi Alaoui, G. Guasp and M. Nicolau, \textit{On deformations of transversely homogeneous foliations}, Topology \textbf{40}(6), p. 1363-1393, 2001.
		%\bibitem{Kacimi1} A. El Kacimi Alaoui, \textit{Foliated cohomology and infinitesimal deformations of developable foliations}, Preprint, 2023. Available at \url{http://perso.numericable.fr/azizelkacimi/}.
		%\bibitem{Kacimi2} A. El Kacimi Alaoui, \textit{Cohomologie des groupes discrets à valeurs dans un Fr\'{e}chet}, Preprint, 2022. Available at \url{http://perso.numericable.fr/azizelkacimi/}.
		\bibitem{thesis} S. Geudens, \textit{On Submanifolds and Deformations in Poisson Geometry}, Ph.D. thesis, KU Leuven, 2021. 
		Available at \url{https://kuleuven.limo.libis.be/discovery/search?vid=32KUL_KUL:KULeuven}.
		\bibitem{Haller} S. Haller and T. Rybicki, \textit{On the group of diffeomorphisms preserving a locally conformal symplectic structure}, Ann. Glob. Anal. Geom. \textbf{17}(5), p. 475-502, 1999.
		\bibitem{heitsch} J.L. Heitsch, \textit{A cohomology for foliated manifolds}, Comment. Math. Helv. \textbf{50}(1), p. 197-218, 1975.
		\bibitem{Istrati} N. Istrati and A. Otiman, \textit{De Rham and twisted cohomology of Oeljeklaus-Toma manifolds}, Ann. Inst. Fourier \textbf{69}(5), p. 2037-2066, 2019.
		\bibitem{Lang} S. Lang, \textit{Algebra}, Revised Third Edition, Graduate Texts in Mathematics, Volume 211, Springer, 2002.
		\bibitem{Lichnerowicz} A. Lichnerowicz, \textit{Les vari\'{e}t\'{e}s de Poisson et leurs algèbres de Lie associ\'{e}es}, J. Differential Geom. \textbf{12}(2), p. 253-300, 1977.
		\bibitem{marcut} I. M\u{a}rcu\c{t} and A. Sch\"{u}ssler, \textit{The Serre spectral sequence of a Lie subalgebroid}, ArXiv:2405.00419, 2024.
		\bibitem{Matsumoto} S. Matsumoto and N. Tsuchiya, \textit{The Lie affine foliations on $4$-manifolds}, Invent. Math. \textbf{109}, p. 1-16, 1992.
		\bibitem{Moerdijk} I. Moerdijk and J. Mr\v{c}un, \textit{Introduction to Foliations and Lie Groupoids}, Cambridge Studies in Advanced Mathematics, Volume 91, Cambridge University Press, 2003. 
		\bibitem{Molino} P. Molino, \textit{Riemannian foliations}, Progress in Mathematics, Volume 73, Birkhäuser Boston, 1988.
		\bibitem{Moroianu} A. Moroianu and M. Pilca, \textit{Closed $1$-forms and twisted cohomology}, J. Geom. Anal. \textbf{31}(8), p. 8334-8346, 2021.
		\bibitem{Nomizu} K. Nomizu, \textit{On the cohomology of compact homogeneous spaces of nilpotent Lie groups}, Ann. of Math. \textbf{52}(3), p. 531-538, 1954.
		\bibitem{Novikov} S.P. Novikov, \textit{The Hamiltonian formalism and a multivalued analogue of Morse theory}, Uspekhi Mat. Nauk \textbf{37}, p. 3-49, 1982.
		\bibitem{Ornea} L. Ornea and M. Verbitsky, \textit{Morse-Novikov cohomology of locally conformally K\"{a}hler manifolds}, J. Geom. Phys. \textbf{59}(3), p. 295-305.
		\bibitem{Otiman} A. Otiman, \textit{Morse-Novikov cohomology of locally conformally K\"{a}hler surfaces}, Math. Z. \textbf{289}(1-2), p. 605–628, 2018.
		\bibitem{Pajitnov} A. Pajitnov, \textit{An analytic proof of the real part of Novikov's inequalities}, Soviet Math. Dokl. \textbf{35}(2), p. 456-457, 1987.
		\bibitem{Palis} J. Palis and J.C. Yoccoz, \textit{Centralizers of Anosov diffeomorphisms on tori}, Ann. Scient. \'{E}c. Norm. Sup. \textbf{4}(22), p. 99-108, 1989.
	\end{thebibliography}
\end{document}